\documentclass[12pt,final]{article}  

\usepackage{amsmath}
\usepackage{amsthm}

\usepackage{amssymb}
\usepackage{amsfonts}
\usepackage{latexsym}               
\usepackage{amsgen}
\usepackage[mathscr]{eucal}
\usepackage{mathrsfs}
\usepackage{enumerate}
\usepackage{cite}
\usepackage{mathptmx} 
\usepackage{color}
\usepackage{fancybox}
\usepackage[dvips]{graphicx}

\newtheorem{lem}{Lemma}[section]
\newtheorem{thm}[lem]{Theorem}
\newtheorem{pro}[lem]{Proposition}

\newtheorem{defn}[lem]{Definition}
\newtheorem{rem}[lem]{Remark}
\newtheorem{asp}[lem]{Assumption}

\newcommand{\bqn}{\begin{equation}}
\newcommand{\eqn}{\end{equation}}
\newcommand{\beqx}{\begin{equation*}}
\newcommand{\eeqx}{\end{equation*}}
\newcommand{\barr}{\begin{array}}
\newcommand{\earr}{\end{array}}
\newcommand{\beqn}{\begin{eqnarray}}
\newcommand{\eeqn}{\end{eqnarray}}
\newcommand{\beqnx}{\begin{eqnarray*}}
\newcommand{\eeqnx}{\end{eqnarray*}}
\newcommand{\bmt}{\begin{multline}}
\newcommand{\emt}{\end{multline}}

\numberwithin{equation}{section}

\newcommand{\D}{\partial}

\newcommand{\bbr}{{\mathbb R}}

\newcommand{\al}{\alpha}
\newcommand{\be}{\beta}
\newcommand{\ga}{{\gamma}}
\newcommand{\Ga}{{\Gamma}}

\newcommand{\ve}{\varepsilon}
\newcommand{\de}{\delta}

\newcommand{\la}{\lambda}

\newcommand{\Om}{\Omega}

\newcommand{\er}{\eqref}
\newcommand{\lb}{\label}

\title{Time-periodic flows of electrons and holes in semiconductor devices}

\author{%
{\large\sc Toru Kan${}^1$} and {\large\sc Masahiro Suzuki${}^2$}}

\date{%
\normalsize
${}^1$%
Department of Mathematical Sciences, Osaka Prefecture University, 
\\
1-1 Gakuen-cho, Naka-ku, Sakai, 599-8531, Japan
\\ [7pt]
${}^2$%
Department of Computer Science and Engineering, 
Nagoya Institute of Technology,
\\
Gokiso-cho, Showa-ku, Nagoya, 466-8555, Japan
}

\begin{document}

\maketitle

\begin{abstract}
The main purpose of this paper is mathematical analysis 
on time-periodic flows of electrons and holes in semiconductors.
The flows appear in a situation 
that alternating-current voltages are applied to devices.
In this paper, we study the drift-diffusion model for semiconductors
in a three-dimensional bounded domain
and investigate the existence and stability of time-periodic solutions.
We first derive the uniform-in-time estimate of time-global solutions,
and then prove by the relative entropy method
that the difference of any two solutions decays exponentially 
fast as time tends to infinity.
These facts enable us to show 
the unique existence and global stability of time-periodic solution.
\end{abstract}

\begin{description}
\item[{\it Keywords:}]
drift-diffusion model; parabolic--elliptic system; initial--boundary value problem; \\
mixed-boundary condition; time-periodic solution; global stability

\item[{\it 2010 Mathematics Subject Classification:}]
35B10; 
35B35; 
35B40; 
35B45; 
82D37  
\end{description}

\newpage

\section{Drift-diffusion model}

This paper is concerned with the asymptotic behavior 
of solutions to the drift-diffusion model for semiconductors.
The model was proposed by Roosbroeck~\cite{W-V.R} as
a system of partial differential equations 
for the transport of electrons and holes in semiconductor devices. 
Let $\Omega \subset \mathbb R^3$ be a domain 
occupied by a semiconductor device.
Then the model is written as the parabolic--elliptic system
\begin{subequations}\label{ddeq}
\begin{equation}\label{ddnd}
\left\{
\begin{aligned}
&n_t =\nabla \cdot (\nabla n -n \nabla v) -R(n,p), & \\
&p_t =\nabla \cdot ( \nabla p +p\nabla v) -R(n,p), & \\
&\ve \Delta v=n-p-D(x), & (t,x) \in I \times \Om,
\end{aligned}
\right.
\end{equation}
where $I \subset \mathbb R$ is an open interval with $\sup I =\infty$.
The unknown functions $n$, $p$ and $v$ stand for the electron density, 
the hole density and the electrostatic potential, respectively.
The recombination-generation term $R$ accounts 
for instantaneous generation or annihilation of electron-hole pairs.
The doping profile $D$ denotes the density of ionized impurities 
in semiconductors, and determines the performance of devices.
The positive constant $\varepsilon$ is the scaled Debye length.
For more details of this model, see \cite{JJ,JA,MRS,NS5}.

We divide the boundary $\partial\Omega$ into two parts $\Gamma_D$ and $\Gamma_N$,
and impose a mixed boundary condition as follows.
\begin{equation}\label{bdnd}
\left\{
\begin{aligned}
&n =N_b(x), \ p=P_b(x), \ v=V_b(t,x) & & (t,x) \in I \times \Gamma_D, \\
&(\nabla n -n\nabla v) \cdot \bold{n}=(\nabla p +p\nabla v) \cdot \bold{n}=0, \
 \varepsilon \nabla v \cdot \bold{n}+b(x)v=g(t,x) && (t,x) \in I \times \Gamma_N.
\end{aligned}
\right.
\end{equation}
\end{subequations}
Here $N_b$, $P_b$, $V_b$, $b$ and $g$ are given functions
and $\bold{n}$ denotes the outer unit normal vector to $\partial \Omega$.
From a physical point of view,
this boundary condition corresponds to Ohmic, Schottky or Metal-Oxide contact
arising in widely used semiconductor devices
such as MOSFETs, p-n diodes, thyristors and so on.

There have been many researches on 
the existence and asymptotic behavior of solutions to 
the initial--boundary value problem of \er{ddeq} 
with time-independent boundary data. 
A pioneer work was made by Mock \cite{M.S.M1,M.S.M2,M.S.M3} 
for the simpler case $\Gamma_D=\emptyset$ and $b=g=0$.
It was shown that a solution exists globally in time and converges to a stationary solution.
Physically speaking, the boundary condition in this case
does not allow any electron and hole to flow through the boundary.
Gajewski and Gr\"oger \cite{GG86} 
proved the time-global solvability for the more relevant case
$\Gamma_D\neq\emptyset$, $g\neq 0$ and $b \geq0$ (see also \cite{Ga3}).
In this case, electrons and holes can flow through the boundary.
Furthermore, they investigated the asymptotic state of solutions 
for a special boundary data $N_b$, $P_b$ and $V_b$,
and then showed the global stability of a special stationary solution $(N,P,V)$ 
which represents a thermal equilibrium, 
that is, 
\begin{equation}\label{noflow}
NP=1, \quad \nabla(\log N-V)=\nabla(\log P+V)=0.
\end{equation}
The second and third equalities mean that the currents vanish,
and therefore their results do not cover physically important situations
that semiconductor devices are used in integrated circuits.

For general time-independent boundary data, 
Gr\"oger \cite{Gr} constructed a stationary solution in which the current is flowing (see also \cite{Ga1,T.S.80}).
Of course, it is of great interest to study its stability.
A difficulty in proving the global stability lies in the derivation of uniform-in-time estimates.
The study \cite{GG86} used  \eqref{noflow} to resolve this difficultly.
For a simpler case $b=g=0$, Fang and Ito \cite{FI1,FI2,FI3} derived the uniform-in-time estimate,
and then constructed a compact attractor.
The relation between the attractor and stationary solutions was not clarified.
In this paper, we first extend their result on the uniform-in-time estimate to the case $b \neq 0$ and $g\neq 0$
assuming 
that the area of $\{ b \neq 0 \}$ is small.
The set $\{ b \neq 0 \}$ corresponds to interfaces between semiconductors and oxides.
It is worth pointing out that the uniform-in-time estimate can be obtained even in the case 
when the current is large and/or the boundary data $g$ and $V_{b}$ are time-dependent.


Besides stationary flows, time-periodic flows are also physically important.
Indeed, time-periodic flows appear when PN junction diodes act like a rectifier 
by converting alternating current into direct current.
In a one-dimensional case, the authors \cite{KS1} studied
the unique existence and global stability of time-periodic solutions
in a situation that the applied voltage is periodic in time.
This time-periodic solution has nonzero currents.
Seidman \cite{T.S.} also investigated time-periodic solutions
for a generalized drift-diffusion model. 
In this paper, we show 
the global stability of time-periodic solutions
for time-periodic boundary data.
Our main theorem also ensures 
the stability of stationary solutions in which small currents are flowing.



\section{Main theorems}

We begin with introducing notation and making assumptions to be used throughout the paper.
For $ 1 \leq q \leq \infty$, 
$|\cdot|_q$ and $|\cdot|_{q,\Gamma_N}$ denote the norms of the Lebesgue spaces 
$L^q(\Omega)$ and $L^q(\Gamma_N)$, respectively.
Furthermore, $\Vert\cdot\Vert_1$ stands for
the norm of the Sobolev space $H^1(\Omega)$.
We denote by $H^1_D(\Omega)$ the subspace
$\{ f \in H^1(\Omega);f=0 \ \mbox{on} \ \Gamma_D\}$
and by $H^1_D(\Omega)^*$ its dual space.
The notation $f'$ means the derivative of a function $f$ with respect to $t$.
For $a,\sigma \in \mathbb{R}$, we write $a_+ :=\max \{ a,0\}$, $a_- :=\min \{ a,0\}$
and $a_{\sigma}:=(a-\sigma)_{+}+\sigma=\max\{a,\sigma\}$.


\begin{asp}\lb{assp1}
We assume conditions (H1)--(H8) below.
\begin{enumerate}
\item[(H1)]
$\Omega \subset \mathbb{R}^3$ is a bounded domain with Lipschitz boundary.
\item[(H2)]
$\partial \Omega$ consists of the disjoint union of $\Gamma_D$ and $\Gamma_N$, 
and the measure of $\Gamma_D$ is nonzero. 
\item[(H3)]
The recombination-generation term $R$ is given by the Shockley-Read-Hall form,
that is,
\begin{equation*}
R(n,p):=c_{R} \frac{np-1}{n+p+2},
\end{equation*}
where $c_R$ is a positive constant.
\item[(H4)]
$N_b=N_b(x)$, $P_b=P_b(x)$, $V_b=V_b(t,x)$, $D=D(x)$, $b=b(x)$ and $g=g(t,x)$ 
are given functions satisfying
$N_b,P_b \in H^1(\Omega) \cap L^\infty (\Omega)$,
$V_b \in W^{1,\infty}(\mathbb R;W^{1,\rho}(\Omega))
\cap L^\infty (\mathbb R \times \Omega)$,
$D \in L^\infty (\Omega)$,
$b \in L^\infty (\Gamma_N)$,
$g \in W^{1,\infty}(\mathbb R;L^r (\Gamma_N))$
for some $\rho>3$ and $r>2$.
\item[(H5)]
$|N_b|_\infty$, $|P_b|_\infty$, $\|N_b\|_1$, $\|P_b\|_1$, 
$\sup_{t \in  \mathbb R}|V_b(t)|_{\infty}$,
$\sup_{t \in  \mathbb R} \| V_b(t)\|_1$, 
$|D|_\infty$,
$|b|_{\infty,\Gamma_N}$, $\sup_{t \in \mathbb R} |g(t)|_{r,\Gamma_N}$ $\le C_0$
for some positive constant $C_0$.
\item[(H6)]
$N_b$, $P_b \ge c_0$ in $\Omega$ with some positive constant $c_0$.
\label{Hlb}
\item[(H7)]
$b \ge 0$ on $\Gamma_N$.
\item[(H8)]
$\displaystyle \de:=\sup_{t \in \mathbb R} (|\nabla \log N_b-\nabla V_b(t)|_\infty^2
+|\nabla \log P_b+\nabla V_b(t)|_\infty^2+\|V_b'(t)\|_{W^{1,\rho}}
+|g'(t)|_{r,\Gamma_N})
+|\log(N_bP_b)|_\infty$ $<\infty.$
\end{enumerate} 
\end{asp}

Let us also give the definition of solutions of \eqref{ddeq}. 

\begin{defn}\lb{DOS}
We say that $(n,p,v)$ is a solution of \eqref{ddeq}
if it satisfies the following conditions. 
\begin{enumerate}
\item[(i)]
For any bounded interval $J \subset I$,
\begin{equation}\lb{reg0}
\begin{gathered}
n-N_b \in L^2(J;H^1_D(\Om)) \cap L^\infty (J \times \Om), 
\quad n' \in L^2(J;H^{1}_D(\Om)^*),
\\
p-P_b \in L^2(J;H^1_D(\Om)) \cap L^\infty (J \times \Om),
\quad p' \in L^2(J;H^{1}_D(\Om)^*),
\\
v-V_b \in C(I;H^1_D(\Om)).
\end{gathered} 
\end{equation}
\item[(ii)] 
$n,p \geq 0$ a.e. in  $I \times \Omega$.
\item[(iii)]
For any $\phi_1,\phi_2,\phi_3 \in H^1_D(\Om)$ and a.e. $t \in I$,
\begin{subequations}\label{ddspw}
\begin{gather}
\langle n',\phi_1 \rangle
+\int_\Om \left\{ (\nabla n -n \nabla v) \cdot \nabla \phi_1 +R(n,p)\phi_1 \right\} dx=0,
\lb{ddewn} \\
\langle p',\phi_2 \rangle
+\int_\Om \left\{ (\nabla p +p \nabla v) \cdot \nabla \phi_2 +R(n,p)\phi_2 \right\} dx=0,
\lb{ddewp} \\
\ve \int_\Om \nabla v \cdot \nabla \phi_3 dx+\int_{\Gamma_N} (bv-g) \phi_3 dS
 =-\int_\Om (n-p-D)\phi_3 dx.
\lb{ddewv}
\end{gather}
\end{subequations}
\end{enumerate}
Furthermore, if $(n,p,v)$ is a solution of \eqref{ddeq} with $I=\bbr$ 
and additionally satisfies the condition (iv) below,
we say that $(n,p,v)$ is a time-periodic solution of \eqref{ddeq} with period $T_*$.
\begin{itemize}
\item[(iv)]
$(n,p,v)(t+T_*,x)=(n,p,v)(t,x)$
for some constant $T_*>0$.
\end{itemize}
\end{defn}

We are now in a position to state our main theorems.
As mentioned above, Gajewski and Gr\"oger~\cite{GG86} 
considered the problem \er{ddeq}
with the initial condition $(n,p)(0,\cdot)=(n_0,p_0)$,
and showed the global existence of a solution $(n,p,v)$
for any nonnegative initial data $(n_0,p_0) \in L^\infty(\Om)\times L^\infty(\Om)$.
The result on the time-global solvability will be introduced precisely in Section~\ref{prelim}.
Our first theorem deals with the uniform-in-time estimates of $(n,p)$. 
Here and subsequently, we fix $q_{0}>2$ and set
\begin{gather*}
\theta:=|b|_{q_{0},\Gamma_{N}}.
\end{gather*}

\begin{thm}\lb{uesthm}
There exist positive constants $\hat \theta$ and $\hat C$ depending only on $\max\{\delta,1\}$,
$c_0$, $C_0$, $\Omega$, $\Gamma_D$, $\rho$, $r$, $\ve$ and $c_R$
such that if $\theta<\hat \theta$, then any solution $(n,p,v)$ of \eqref{ddeq} satisfies
\begin{equation}\lb{gsulb}
\limsup_{t \to \infty}\left( \left| \frac{1}{n(t)} \right|_\infty +\left| \frac{1}{p(t)} \right|_\infty
 +|n(t)|_\infty +|p(t)|_\infty \right) \le \hat C.
\end{equation}
\end{thm}

We are interested in the asymptotic behavior of solutions
in the case that the applied voltage is periodic in time as an AC voltage.
We prove that if $\delta$ is sufficiently small, then
\er{ddeq} has a unique time-periodic solution
and any solution converges to it as $t \to \infty$.
This result is summarized in the following theorem. 

\begin{thm}\lb{tpsthm}
Suppose that $V_b$ and $g$ are periodic in $t$ with period $T_*>0$,
and that $\theta<\hat \theta$ 
holds for $\hat \theta$ being in Theorem \ref{uesthm}.
Then there exists a constant $\delta_0>0$ depending only on 
$c_0$, $C_0$, $\Omega$, $\Gamma_D$, $\rho$, $r$, $\ve$ and $c_R$
such that \eqref{ddeq} has a unique time-periodic solution $(n_*,p_*,v_*)$ if $\delta<\delta_0$.
Furthermore, 
any solution $(n,p,v)$ of \eqref{ddeq} converges to $(n_*,p_*,v_*)$
in $L^2(\Om) \times L^2(\Om) \times H^1(\Omega)$ 
exponentially fast as $t \to \infty$. Specifically,
\begin{equation}\lb{contps}
\limsup_{t \to \infty} e^{ct} (|(n-n_*)(t)|_2+|(p-p_*)(t)|_2+\|(v-v_*)(t)\|_1)<+\infty,
\end{equation}
where $c>0$ is a constant depending only on 
$c_0$, $C_0$, $\Omega$, $\Gamma_D$, $\rho$, $r$, $\ve$ and $c_R$.
\end{thm}

\begin{rem} \rm
Theorem \ref{tpsthm} provides the global stability of asymptotic states which do not satisfy \er{noflow}, 
and therefore it is an extension of \cite{GG86} for the case $\theta \ll 1$.
The smallness assumption on $\theta$ holds if the area $S_{O}$ of the set $ \{ b \neq 0 \} $ is small,
because of the inequality $|b|_{q_{0},\Gamma_{N}} \leq |b|_{\infty,\Gamma_{N}} S_{O}^{1/q_{0}}$.
From a physical point of view, $S_{O}$ denotes the areas of interfaces between semiconductors and oxides.
We also remark that stationary solutions may not be unique for large $\delta$
(see \cite{R87,S89a,S89b}). 
\end{rem}


The main task in the proof of Theorem 2.3 is to establish the uniform-in-time estimate of $L^1$-norms. 
Then a difficulty arises from the condition $b \neq 0$ in \eqref{ddewv}. 
Specifically, when we rewrite the drift terms $n\nabla v$ and $p\nabla v$ in \eqref{ddewn} and \eqref{ddewp} 
by using \eqref{ddewv}, 
we have boundary terms having a strong nonlinearity. 
To handle the difficulty, we decompose $n$ and $p$ into the parts of the lower and higher values. 
The lower part is not issue at all. 
If the $L^1$-norm of the higher part is large, 
the dissipative effect is strong enough so that 
the rewrite mentioned above is not necessary. 
For the case that the $L^1$-norm of the higher part is small, we use a new technique 
to show that the nonlinear term with $b$ can be absorbed into the dissipative terms. 
It will be discussed in Lemma \ref{lemefs}. 

A main difficulty of the proof of Theorem 2.4 arises due to the low regularity of solutions.
It is not expected that the solutions have a better regularity than \eqref{reg0}
due to the mixed boundary condition \eqref{bdnd}. 
In such a case, it is not straightforward to handle some nonlinear terms
only by applying the well-known inequalities. 
To overcome this difficulty, we use the new estimate proved in Lemma \ref{mpees}.

This paper is organized as follows.
In Section~\ref{prelim}, we introduce basic inequalities and facts.
Section~\ref{ULBS} is devote to the derivation of 
uniform-in-time estimates in Theorem~\ref{uesthm}.
In Section~\ref{energy}, we estimate the difference of two solutions.
The estimate enables us to prove Theorem~\ref{tpsthm}.
The proof of Theorem~\ref{tpsthm} will be discussed in Section~\ref{TPS}.

\section{Preliminaries}\lb{prelim}
In this section, 
we introduce basic facts to be used in our arguments.
First we mention the time-global solvability of the problem \eqref{ddeq} 
supplemented with the initial condition $(n,p)(0,\cdot)=(n_{0},p_{0})$. 
It is proved by Gajewski and Gr\"oger~\cite[Theorem~1]{GG86}
in the case that the function $V_{b}$ is time-independent.
Even if $V_{b}$ is time-dependent, 
their proof works under the condition (H4) in Assumption~\ref{assp1}.
\begin{pro}[\!\!\cite{GG86}]\label{global_GG86}
Let $I=(0,\infty)$. 
For any $(n_{0},p_{0}) \in L^{\infty}(\Omega) \times L^{\infty}(\Omega)$ with $n_0,p_0 \ge 0$,
the problem \eqref{ddeq} has a unique solution $(n,p,v)$ satisfying
$(n,p)(t,\cdot ) \to (n_0,p_0)$ in $L^2(\Om)\times L^2(\Om)$ as $t \downarrow 0$.
\end{pro}

Next we discuss estimates of a solution $w$ of the boundary value problem 
\begin{equation}\label{eeq}
\left\{
\begin{aligned}
&\Delta w=h(x), && x \in \Om, \\
&w=W_b(x), && x \in \Gamma_D, \\
&\nabla w \cdot \bold{n}+\tilde b(x)w=\tilde g(x), && x \in \Gamma_N,
\end{aligned}
\right.
\end{equation}
where $h \in L^{6/5}(\Omega)$, $W_b \in H^1(\Omega)$, 
$\tilde b \in L^2 (\Gamma_N)$, $\tilde b \ge 0$ 
and $\tilde g \in L^{4/3}(\Gamma_N)$.
We say that $w \in H^1(\Omega)$ is a solution of \eqref{eeq} 
if $w-W_b \in H^1_D(\Omega)$ and
\begin{equation*}
\int_\Om \nabla w \cdot \nabla \phi dx+\int_{\Gamma_N} ({\tilde b}w-\tilde g) \phi dS
 =-\int_\Om h\phi dx
\end{equation*}
for all $\phi \in H^1_D(\Omega)$.
This is written as
\begin{equation}
\int_\Om \nabla (w-W_b) \cdot \nabla \phi dx+\int_{\Gamma_N} {\tilde b}(w-W_b) \phi dS
 =-\int_\Om \nabla W_b \cdot \nabla \phi dx -\int_\Om h\phi dx
  -\int_{\Gamma_N} ({\tilde b}W_b -{\tilde g}) \phi dS.
\label{weakeq}
\end{equation}
The Sobolev embedding theorem $H^1(\Omega) \hookrightarrow L^6(\Omega)$
and the inequality \eqref{tsineq} below imply that the mapping 
$\phi \mapsto -\int_\Om \nabla W_b \cdot \nabla \phi dx
-\int_\Om h\phi dx +\int_{\Gamma_N} ({\tilde b}W_b -{\tilde g}) \phi dS$
defines a bounded linear functional on $H^1_D(\Om)$.
Therefore a standard theory of elliptic partial differential equations
shows that \eqref{eeq} has a unique solution $w$ satisfying
\begin{equation}
\|w-W_b\|_1 \le C\{|h|_{6/5}+|\tilde g|_{4/3,\Gamma_{N}}
 +\|W_{b}\|_{1}(1+|\tilde b|_{2,\Gamma_{N}})\}
\label{ell-H1} 
\end{equation}
for some constant $C=C(\Omega,\Gamma_D)>0$.

\begin{lem}\label{eespro}
Let $1 \le \al<3/2$ and $1 \le \be<2$.
Then the solution $w$ of \eqref{eeq} satisfies
\begin{equation}
|\nabla w|_{\alpha} +|w|_{\beta,\Gamma_{N}} \le C\{|h|_{1}
 +|\tilde g|_{1,\Gamma_{N}} +\|W_{b}\|_{1}(1+|\tilde b|_{4/3,\Gamma_{N}})\},
\label{ell-W11Lp}
\end{equation}
where $C=C(\Omega,\Gamma_D,\alpha,\beta)>0$ is a constant.
Suppose further that $h \in L^q(\Omega)$, 
$W_b \in L^\infty(\Omega)$ and $\tilde g \in L^r(\Gamma_N)$
for some $q>3/2$ and $r>2$.
Then there exists a constant $\tilde C=\tilde C(\Omega,\Gamma_D,q,r)>0$
such that $w \in L^\infty(\Omega)$ and
\begin{equation}
|w|_\infty \le \tilde C (|h|_q+|\tilde g|_{r,\Gamma_{N}} +|W_b|_\infty).
\label{ell-Linf}
\end{equation}
\end{lem}

The following lemmas will also be utilized.

\begin{lem}\lb{Traces}
The following hold.
\begin{enumerate}[(i)]
\item
For $1 \le q \le 2$,
there exists a constant $C=C(\Omega,\Gamma_D,q)>0$ such that
for all $f \in H^1_D(\Omega)$,
\begin{equation}\label{tpineq}
|f|_{q,\Gamma_N} \le C|\nabla f|_q.
\end{equation}
\item
There exists a constant $C=C(\Omega)>0$ such that
for all $f \in H^1(\Omega)$,
\begin{equation}\label{tsineq}
|f|_{4,\Gamma_N} \le C\| f\|_1.
\end{equation}
\item
For $1 \le q<4$, 
there exist constants $C=C(\Omega,\Gamma_D,q)>0$ and $\al=\al(q)>0$ such that
for all $f \in H^1_D(\Omega)$ and $\mu>0$,
\begin{equation}\label{trineq}
|f|_{q,\Gamma_N} \le \mu |\nabla f|_2 +C\mu^{-\al}|f|_1.
\end{equation}
\end{enumerate}
\end{lem}


\begin{lem}\lb{A5}
Let $d \ge 1$.
Then there is a constant $C=C(d)>0$ such that
\begin{equation}\label{efeineq}
\int_A^a \log \frac{y}{A} dy
 +\int_B^b \log \frac{y}{B} dy
  \le C\left\{ (a-b)^2+\frac{ab-1}{a+b+2}\log (ab) +1 \right\}
\end{equation}
for all $a,b>0$ and $1/d \le A,B \le d$.
\end{lem}

\begin{lem}\label{A3}
Let $\sigma>0$.
For $M>e$ and $s \in \bbr$, we set 
\begin{equation*}
h(M):=\frac{1}{2} \log \log M,
\qquad
H_M(s):=\left\{
\begin{aligned}
&-h(M) && (s<-h(M)),
\\
&s && (|s|\leq h(M)),
\\
&h(M) && (s>h(M)).
\end{aligned}
\right.
\end{equation*}
Then there is a constant $M_*=M_*(\sigma)>e$ such that
\begin{equation}
(a+b) \chi_{\{a+b \geq M \}} \leq 
\frac{2}{h(M)}\left\{ (a-b) H_M \left(\log \frac{a_{\sigma}}{b_{\sigma}}\right) 
 +\frac{ab-1}{a+b+2} \log (ab) \right\}
\label{abhineq}
\end{equation}
for all $a,b>0$ and $M \ge M_*$.
\end{lem}

\begin{lem}\lb{A4}
Let $E$ be a measurable set in a Euclidean space.
Assume that sequences 
$\{f_k\} \subset L^2(E)$ and $\{g_k\} \subset L^2(E)$ 
are convergent in $L^2(E)$ 
and that $|f_k| \le C$ in $E$ 
for some constant $C>0$ independent of $k$.
Then $\{ f_k g_k\}$ is convergent in $L^2(E)$.
\end{lem}

We will give the proofs of Lemmas \ref{eespro}--\ref{A4} in Appendix.

\section{Upper and lower bounds of solutions}\lb{ULBS}

Throughout this section, 
we assume that $(n,p,v)$ is a solution of  \eqref{ddeq}. 
The goal of this section is to prove Theorem~\ref{uesthm}.

We set 
\begin{gather*}
{\cal E}_1:=\int_{N_b}^{n} \log \frac{y}{N_b} dy
 +\int_{P_b}^{p} \log \frac{y}{P_b} dy,
\quad {\cal E}_2:=\frac{\ve}{2} |\nabla (v-V_b)|^2, 
\quad {\cal E}_3:=\frac{1}{2}b(v-V_b)^2.
\end{gather*}
Note that these are all nonnegative.
The upper and lower bounds of $n$ and $p$ follow from the following propositions
which will be proved in Subsections~\ref{S4.1}--\ref{S4.3}.

\begin{pro}\label{asulob}
There exist constants $\theta_{0}>0$ and $C>0$ depending only on 
$\max \{ \de,1\},c_0,C_0,\Omega$,
$\Gamma_D,\rho,r,\varepsilon$ and $c_R$
such that if $\theta<\theta_{0}$, then
\begin{equation*}
\limsup_{t \to \infty}\left\{ \int_\Om ({\cal E}_1+{\cal E}_2) dx 
 +\int_{\Gamma_N} {\cal E}_3 dS\right\} \le C.
\end{equation*}
\end{pro}

\begin{pro}\label{libblb}
Suppose that
\begin{equation*}
L:=\limsup_{t \to \infty} (|n(t)|_1 +|p(t)|_1 +\|v(t)\|_1)<\infty.
\end{equation*}
Then there exists a constant 
$C=C(L,c_0,C_0,\Omega,\Gamma_D,r,\varepsilon,c_R)>0$ such that
\begin{equation*}
\limsup_{t \to \infty} (|n(t)|_\infty +|p(t)|_\infty) \le C.
\end{equation*}
\end{pro}

\begin{pro}\label{asubfb}
Suppose that
\begin{equation*}
\tilde L:=\limsup_{t \to \infty} \|v(t)\|_1 <\infty.
\end{equation*}
Then there exists a constant 
$c=c(\tilde L,c_0,C_0,\Omega,\Gamma_D,r,\varepsilon,c_R)>0$ such that
\begin{equation*}
\liminf_{t \to \infty}\inf_\Omega n(t,\cdot) \ge c, \qquad
\liminf_{t \to \infty}\inf_\Omega p(t,\cdot) \ge c.
\end{equation*}
\end{pro}


For a moment, we assume that 
Propositions~\ref{asulob}--\ref{asubfb} hold,
and then complete the proof of Theorem~\ref{uesthm}.

\begin{proof}[Proof of Theorem~\ref{uesthm}]
It is elementary to show that ${\cal E}_1 \ge (n-eN_b)_+ +(p-eP_b)_+$.
From this and the Poincar\'{e} inequality, 
we have
\begin{equation*}
|(n-eN_b)_+|_1 +|(p-eP_b)_+|_1 +\|v-V_b\|_1
\le \int_\Om {\cal E}_1dx +C\left( \int_\Om {\cal E}_2dx \right)^{1/2},
\end{equation*}
where $C=C(\Om,\Gamma_D,\ve)>0$ is a constant.
Combining this with Proposition~\ref{asulob} implies that
\begin{equation*}
\limsup_{t \to \infty} \left( |n|_1 +|p|_1 +\|v\|_1 \right) \le C
\end{equation*}
for some constant $C=C(\max \{ \de,1\},c_0,C_0,\Omega,\Gamma_D,\rho,r,\varepsilon,c_R)>0$.
This together with Propositions~\ref{libblb} and \ref{asubfb} leads to \eqref{gsulb}.
\end{proof}

\subsection{Proof of Proposition~\ref{asulob}}\lb{S4.1}

This subsection is devoted to the proof of Proposition~\ref{asulob}.
We begin by deriving an energy inequality. We put
\begin{gather*}
{\cal D}_1:=n\left| \nabla \log {n}-\nabla v\right|^2
 +p\left| \nabla \log {p}+\nabla v \right|^2,
  \quad {\cal D}_2:=R(n,p)\log (np).
\end{gather*}
\begin{lem}\label{EnergyEq}
There holds that 
\begin{gather}
n,p>0 \quad  \text{a.e. in $I \times \Omega$}.
\label{positive1}
\end{gather}
Furthermore, 
there is a constant $C_1=C_1(C_0,\Omega,\Gamma_D,\rho,r,\varepsilon,c_R)>0$ such that
\begin{equation}
 \frac{d}{dt}\left\{ \int_\Om ({\cal E}_1 +{\cal E}_2) dx +\int_{\Gamma_N} {\cal E}_3 dS \right\}
 +\frac{3}{4} \int_\Om ({\cal D}_1+{\cal D}_2) dx
  \le  C_1 \delta \int_\Omega (n+p) dx +C_1\de.
\label{emineq0}
\end{equation} 
\end{lem}
\begin{proof}
We first show that for $t,t_0 \in I$,
\begin{align}
&\left.\left\{  \int_\Om ({\cal E}_1 +{\cal E}_2)dx +\int_{\Gamma_N} {\cal E}_3 dS \right\}\right|_{t_0}^t
 +\int_{t_0}^t\int_\Om ({\cal D}_1+{\cal D}_2) dxdt
\notag \\
& =\int_{t_0}^t\int_\Om (I_{1}+I_{2}+I_{3}) dxdt +\int_{t_0}^t\int_{\Gamma_N} I_{4} dSdt,
\label{beeq} 
\end{align}
where
\begin{gather*}
I_{1}:=(\nabla n - n\nabla v) \cdot ( \nabla \log {N_b} -\nabla V_b)
  +(\nabla p + p\nabla v) \cdot ( \nabla \log {P_b} +\nabla V_b ),
\notag \\
I_{2}:=R(n,p) \log (N_bP_b),
 \quad I_{3}=- \ve\nabla V_b' \cdot \nabla (v-V_b),
  \quad I_{4}:=-(bV_b'-g') (v-V_b).
\notag
\end{gather*}

To make the computation in this paragraph rigorous, 
we use a mollifier with respect to the time variable $t$ 
due to the insufficiency of the regularity of solutions.
We omit the argument since it is standard.
We differentiate \eqref{ddewv} with respect to $t$ to find that
\begin{equation*}
\ve \int_\Om \nabla v' \cdot \nabla \phi_3 dx+\int_{\Gamma_N} (bv'-g') \phi_3 dS
 =-\langle n'-p',\phi_3 \rangle.
\end{equation*}
Combining \eqref{ddewn}, \eqref{ddewp} and this equality gives
\begin{align}
&\langle n', \phi_1+\phi_3 \rangle +\langle p', \phi_2-\phi_3 \rangle
 +\ve \int_\Om \nabla (v-V_b)' \cdot \nabla \phi_3 dx+\int_{\Gamma_N} b(v-V_b)' \phi_3 dS 
 \notag \\
&\quad +\int_\Om n ( \nabla \log {n} -\nabla v ) \cdot \nabla \phi_1 dx
 +\int_\Om p (\nabla \log {p} +\nabla v ) \cdot \nabla \phi_2 dx 
 +\int_\Om R(n,p) (\phi_1+\phi_2) dx 
 \notag \\
&= -\ve \int_\Om \nabla V_b' \cdot \nabla \phi_3 dx
 -\int_{\Gamma_N} (bV_b'-g') \phi_3 dS.
\label{beeq0}
\end{align}
Let $0<\sigma \le c_0$,
where $c_0$ is the constant in the assumption (H$6$).
Taking $\phi_1=\log (n_{\sigma}/N_b)-(v-V_b)$,
$\phi_2=\log (p_{\sigma}/P_b)+(v-V_b)$ and $\phi_3=v-V_b$ 
and integrating the result over $[t_0,t]$,
we have
\begin{align}
&\left.\left\{ \int_\Om \left(\int_{N_b}^{n} \log \frac{y_\sigma}{N_b} dy +\int_{P_b}^{p} \log \frac{y_\sigma}{P_b} dy +{\cal E}_2 \right)  dx 
+\int_{\Gamma_N} {\cal E}_3 dS \right\}\right|_{t_0}^t
\notag \\
&\quad +\int_{t_0}^t\int_\Om (n|\nabla \log n_{\sigma} |^{2}+p|\nabla \log p_{\sigma}|^{2} 
-2\nabla (n-p)\cdot \nabla v + (n+p) |\nabla v|^{2}) dxdt
\notag \\
&\qquad + \int_{t_0}^t\int_\Om c_R \frac{n_{\sigma}p_{\sigma}-1}{n+p+2} \log (n_{\sigma}p_{\sigma}) dxdt
\notag \\
& = \int_{t_0}^t\int_\Om (I_{1}+I_{2}+I_{3}) dxdt +\int_{t_0}^t\int_{\Gamma_N} I_{4} dSdt
+ \int_{t_0}^t\int_\Om (J_{1}^{\sigma} + J_{2}^{\sigma}) dxdt,
\label{beeq0} 
\end{align}
where $a_{\sigma}=\max\{a,\sigma\}$ for $a \in \bbr$,
\begin{gather*}
J_{1}^{\sigma}:= -(\nabla n -\nabla n_\sigma) \cdot \nabla v 
+(\nabla p -\nabla p_\sigma) \cdot \nabla v,
\qquad
J_{2}^{\sigma}:= c_R \frac{n_{\sigma}p_{\sigma}-np}{n+p+2} \log (n_{\sigma}p_{\sigma}).
\end{gather*}

We now derive \eqref{beeq} by letting $\sigma \to 0$ in \eqref{beeq0}.
The monotone convergence theorem shows that the left-hand side of \eqref{beeq0}
converges to that of \eqref{beeq}.
One sees that $J_1^{\sigma} \to 0$ a.e. in $(t_0,t) \times \Omega$ 
and $|J_{1}^{\sigma}| \leq (|\nabla n| + |\nabla p|)|\nabla v |$.
Furthermore, $J_{2}^{\sigma}$ is estimated as
\begin{align*}
|J_{2}^{\sigma}| &\leq c_R \frac{(\sigma -p)n}{n+p+2} |\log (\sigma n)| \chi_{\{ p<\sigma \le n\}}
 +c_R \frac{(\sigma -n)p}{n+p+2} |\log (\sigma p)| \chi_{\{ n<\sigma \le p\}}
\\
&\quad  +c_R \frac{\sigma^2 -np}{n+p+2} |\log (\sigma^2)| \chi_{\{ n,p<\sigma \}}
\\
&\leq \frac{c_R}{2} \sigma n |\log (\sigma n)| +\frac{c_R}{2} \sigma p |\log (\sigma p)|
 +c_R \sigma^2 |\log \sigma|.
\end{align*}
It follows from the dominated convergence theorem that 
the rightmost term of \eqref{beeq0} converges to $0$.
Thus \eqref{beeq} is verified.
In particular, we have $\int_{\Omega} {\cal D}_{2}  dx <\infty$, 
and therefore \eqref{positive1} holds.

To derive \eqref{emineq0}, we estimate the right-hand side of \eqref{beeq}.
The Schwarz inequality yields
\begin{align}
I_1 & =  n(\nabla \log n -\nabla v) \cdot ( \nabla \log {N_b} -\nabla V_b)
  +p(\nabla \log p +\nabla v) \cdot ( \nabla \log {P_b} +\nabla V_b )
\notag  \\
&\le n \left( \frac{1}{4} \left| \nabla \log {n} -\nabla v \right|^{2}
 +|\nabla \log N_b -\nabla V_b|^2 \right)
+p \left( \frac{1}{4} \left| \nabla \log {p} +\nabla v \right|^{2}
 +|\nabla \log P_b +\nabla V_b|^2 \right)
\notag \\
&\le \frac{1}{4} {\cal D}_1 +\de (n+p). 
\label{beineq1}
\end{align}
It is elementary to show that $|R(n,p)| \le c_R (n+p+1)$.
From this, we see that
\begin{equation}
I_2 \le c_R \de (n+p+1).
\label{beineq2}
\end{equation}
Let $\al \in [1,3/2) $ and $\be \in [1,2)$ 
be the H\"{o}lder conjugates of $\rho>3$ and $r>2$, respectively.
By the H\"{o}lder inequality and \eqref{ell-W11Lp},
we have
\begin{equation}
\int_\Om I_3 dx \le \ve |\nabla V_b'|_\rho |\nabla (v-V_b)|_{\al}
 \le C\de (|n-p|_{1}+1).
\label{beineq3}
\end{equation} 
Using \eqref{ell-W11Lp}, \eqref{tsineq} and the H\"{o}lder inequality,
we deduce that
\begin{align}
\int_{\Gamma_N} I_4 dS &\le |bV_b'|_{4,\Gamma_N} |v-V_b|_{4/3,\Gamma_N}
 +|g'|_{r,\Gamma_N} |v-V_b|_{\be,\Gamma_N}
\notag \\
&\le C (\|V_b'\|_{1} +|g'|_{r,\Gamma_N}) (|n-p|_{1}+1)
\notag \\
&\le C\de (|n-p|_{1}+1).
\label{beineq4}
\end{align}
Here $C>0$ is a constant
depending only on $C_0,\Omega,\Gamma_D,\rho, r$ and $\varepsilon$.

The equality \er{beeq} implies that 
$\int_\Om ({\cal E}_1 +{\cal E}_2) dx +\int_{\Gamma_N} {\cal E}_3 dS$ 
is absolutely continuous in $t$.
Differentiating \er{beeq} in $t$ and then 
plugging \eqref{beineq1}--\eqref{beineq4} into the result,
we obtain the desired inequality.
\end{proof}

We show the following two lemmas to be used in the proof of Proposition~\ref{asulob}.
A key of the proof is to decompose $n$ and $p$ into the parts of the lower and higher values. 
To do so, we introduce a function
\begin{equation*}
{\cal I}^M:=(n+p) \chi_{\{ n+p \ge M\}}, 
\end{equation*}
where $M\geq 0$ and $\chi_A$ denotes the indicator function of a set $A$.


\begin{lem}\label{ulob}
There exist positive constants
$\theta_{1}=\theta_{1}(C_0,\Omega,\Gamma_D,\varepsilon,c_R)$,
$\tilde C_{1}=\tilde C_{1}(C_0,\Omega,\Gamma_D,\varepsilon,c_R)$ and $M_1=M_1(C_0)$
such that if $\theta<\theta_{1}$ and $M \ge M_1$, then
\begin{equation}\label{emineq2}
\int_\Om {\cal I}^M dx \le \tilde C_1\left( \theta +\frac{1}{h(M)}\right)
 \left\{ \int_\Om ({\cal D}_1+{\cal D}_2) dx+1 \right\}.
\end{equation}
Here $h(M)$ is defined in Lemma \ref{A3}.
\end{lem}
\begin{proof}
In the proof,
$C$ denotes a positive constant depending only on
$C_0,\Omega,\Gamma_D,\varepsilon$ and $c_R$.
Set $\sigma:=\max\{|N_{b}|_{\infty}, |P_{b}|_{\infty}\}$ and let $M \ge M_1:=M_*(\sigma)$,
where $M_*(\sigma)$ is the constant being in Lemma~\ref{A3}.
We first claim that it suffices to show the inequality
\begin{equation}\label{lnpineq}
\int_\Omega  (n-p) H_M \left(\log \frac{n_{\sigma}}{p_{\sigma}}\right) dx
 \le C\left( \int_\Omega \mathcal{D}_1 dx +1\right) 
  +C\theta h(M) \left( \int_\Omega {\cal I}^0 dx +1\right),
\end{equation}
where $H_M$ is the function defined in Lemma \ref{A3}.
This claim is verified as follows.
We see from \eqref{abhineq} and \eqref{lnpineq} that
\begin{align}
\int_\Om {\cal I}^M dx 
 &\le \frac{2}{h(M)} \int_\Omega  \left\{ (n-p) H_M \left(\log \frac{n_{\sigma}}{p_{\sigma}}\right)
  +\frac{np-1}{n+p+2} \log (np) \right\} dx
\notag \\
&\le \frac{C}{h(M)} \left\{ \int_\Omega (\mathcal{D}_1 + \mathcal{D}_2)dx +1\right\}
 +C\theta \left( \int_\Omega {\cal I}^0 dx +1 \right).
\label{lnpineq2}
\end{align}
Choosing $M=M_1$ in \eqref{lnpineq2} yields
\begin{equation*}
\int_\Omega {\cal I}^0 dx \le \int_\Om {\cal I}^{M_1} dx +M_1 |\Om|
 \le C\left\{ \int_\Omega (\mathcal{D}_1 + \mathcal{D}_2)dx +1\right\}
  +C\theta \left( \int_\Omega {\cal I}^0 dx +1\right).
\end{equation*}
Hence we have
\begin{equation*}
\int_\Omega {\cal I}^0 dx
 \leq C\left\{ \int_\Omega (\mathcal{D}_1 + \mathcal{D}_2)dx +1\right\},
\end{equation*}
provided that $\theta$ is less than some constant 
$\theta_1=\theta_{1}(C_0,\Omega,\Gamma_D,\varepsilon,c_R)>0$.
Plugging this into \eqref{lnpineq2},
we obtain \eqref{emineq2}.
Thus the claim is proved.

Let us complete the proof by showing \er{lnpineq}. 
We take $\phi_{3}=H_M (\log (n_{\sigma}/p_{\sigma})) \in H^{1}_{D}(\Omega)$ in \eqref{ddewv} to obtain
\begin{gather}
\int_\Omega (n-p) H_M \left(\log \frac{n_{\sigma}}{p_{\sigma}}\right) dx 
 =\varepsilon \int_\Omega I_{1} dx +\int_{\Gamma_N} I_{2} dS
  +\int_{\Gamma_N} I_{3} dS +\int_\Omega I_{4} dx,
\label{L1es} \\
\begin{aligned}
&I_{1}:=-\nabla v \cdot \nabla H_M \left(\log \frac{n_{\sigma}}{p_{\sigma}}\right),
&&I_{2}:=-bv H_M \left(\log \frac{n_{\sigma}}{p_{\sigma}}\right), 
\\
&I_{3}:=g H_M \left(\log \frac{n_{\sigma}}{p_{\sigma}}\right),
&&I_{4}:=D H_M \left(\log \frac{n_{\sigma}}{p_{\sigma}}\right).
\end{aligned}
\notag
\end{gather}
Notice that $I_{1}$ is written as
\begin{align*}
I_1&= \frac{1}{2} \left( \left| \nabla \log n -\nabla v \right|^2 \chi_{\{n > \sigma\}}
 +\left| \nabla \log p+\nabla v \right|^2 \chi_{\{p > \sigma\}} \right)
  \frac{dH_M}{ds} \left(\log \frac{n_{\sigma}}{p_{\sigma}}\right)
\\
&\quad -\frac{1}{2} \left\{ \left| \nabla \log {n_{\sigma}}\right|^2 
 +{\left| \nabla \log {p_{\sigma}} \right|^2}
  +|\nabla v|^2 \left( \chi_{\{n > \sigma\}}+ \chi_{\{p > \sigma\}}\right) \right\}
   \frac{dH_M}{ds} \left(\log \frac{n_{\sigma}}{p_{\sigma}}\right).
\end{align*}
From this and the fact that $0 \le dH_M/ds \le 1$, 
we have
\begin{align}
\int_\Omega I_{1}dx & \leq \frac{1}{2} \int_\Omega 
 \left( \frac{n}{\sigma} \left| \nabla \log n -\nabla v \right|^2 \chi_{\{n > \sigma\}}
  +\frac{p}{\sigma} \left| \nabla \log p+\nabla v \right|^2 \chi_{\{p > \sigma\}} \right) dx
\notag \\
&\quad -\frac{1}{2} \int_\Omega 
 \left( \left| \nabla \log {n_{\sigma}}\right|^2 +\left| \nabla \log {p_{\sigma}} \right|^2 \right)
  \frac{dH_M}{ds} \left(\log \frac{n_{\sigma}}{p_{\sigma}}\right) dx
\notag \\
&\leq \frac{1}{2\sigma} \int_{\Omega} {\cal D}_{1} \,dx
-\frac{1}{2} \int_\Omega \left( \left| \nabla \log {n_{\sigma}}\right|^2 
 +{\left| \nabla \log {p_{\sigma}} \right|^2} \right) 
  \frac{dH_M}{ds} \left(\log \frac{n_{\sigma}}{p_{\sigma}}\right) dx.
\label{L^1es3}
\end{align}
By the fact that $|H_M| \le h(M)$, the H\"{o}lder inequality and \eqref{ell-W11Lp},
the integral of $I_2$ is estimated as
\begin{equation}
\int_{\Gamma_N} I_{2} dS \leq h(M) \theta |v|_{\be,\Gamma_N} 
 \leq C\theta h(M) \left(\int_{\Omega} {\cal I}_{0} \,dx + 1\right),
\label{L^1es4}
\end{equation}
where $\be \in [1,2)$ is the H\"{o}lder conjugate of $q_0>2$.
For the third and fourth terms of the right-hand side of \eqref{L1es},
we utilize the Schwarz and Poincar\'{e} inequalities and \eqref{tpineq} to obtain
\begin{align}
\int_{\Gamma_N} I_{3} dS +\int_{\Omega} I_{4} dx &\le \int_{\Gamma_N} 
 \left( \mu H_M \left(\log \frac{n_{\sigma}}{p_{\sigma}}\right)^2 +\frac{1}{4\mu} g^2 \right) dS 
  +\int_{\Omega} \left( \mu H_M \left(\log \frac{n_{\sigma}}{p_{\sigma}}\right)^2
   +\frac{1}{4\mu} D^2 \right) dx
\notag \\
&\le C\mu \int_\Omega
 \left| \nabla H_M \left(\log \frac{n_{\sigma}}{p_{\sigma}}\right) \right|^2 dx +\frac{C}{\mu}
\notag \\
&\le C\mu \int_\Omega \left( \left| \nabla \log {n_{\sigma}}\right|^2 
 +{\left| \nabla \log {p_{\sigma}} \right|^2} \right) 
  \frac{dH_M}{ds} \left(\log \frac{n_{\sigma}}{p_{\sigma}}\right) dx +\frac{C}{\mu},
\label{L^1es5}
\end{align}
where $\mu>0$ is an arbitrary number.
Plugging \eqref{L^1es3}--\eqref{L^1es5} into \eqref{L1es} and then taking $\mu$ appropriately, 
we obtain \eqref{lnpineq}.
Thus the lemma follows.
\end{proof}


\begin{lem}\label{lemefs}
There exist constants $c_1>0$ and $C>0$ depending only on
$c_{0},C_0,\Om,\Gamma_D,r,\ve$ and $c_R$ such that if 
\begin{equation*}
M \ge \max \{ |N_b|_\infty, |P_b|_\infty \}
\quad
\textrm{and}
\quad
\int_\Omega {\cal I}^M dx \le c_1,
\end{equation*}
then 
\begin{equation*}
\int_\Om ({\cal E}_1 +{\cal E}_2) dx +\int_{\Gamma_N} {\cal E}_3 dS
\le C \left\{ \int_\Omega (\mathcal{D}_1 + \mathcal{D}_2)dx +M^2 +1\right\}.
\end{equation*}
\end{lem}
\begin{proof}
Let $M \ge \max \{ |N_b|_\infty, |P_b|_\infty \}$.
In the proof,
$C$ denotes a generic positive constant depending only on 
$c_{0}$, $C_0,\Om,\Gamma_D,r,\ve$ and $c_R$.
We set
\begin{equation*}
J:=\int_{n \ge M} \frac{|\nabla n|^2}{n} dx+\int_{p \ge M} \frac{|\nabla p|^2}{p} dx
 +\int_\Om (n-p)^2 dx.
\end{equation*}
The assertion follows
if we show the inequality
\begin{equation}\label{ctineq}
J \le C\left( \int_\Om {\cal I}^M dx \right)^{1/2} J
 +C\left( \int_\Om {\cal D}_1 dx +\int_\Om {\cal I}^M dx+M^2+1\right).
\end{equation}
Indeed, this inequality gives
\begin{gather*}
J \le C \left( \int_\Om {\cal D}_1dx +M^2+1\right),
\end{gather*}
provided that $\int_\Omega {\cal I}^M dx$ is sufficiently small.
Combining this with \eqref{ell-H1}, \eqref{tpineq} and \eqref{efeineq},
we see that
\begin{gather*}
\int_{\Omega}{\cal E}_{1}dx 
\leq C\left[ \int_\Omega \{ (n-p)^2 + {\cal D}_{2} \} dx +1\right]
\le C\left\{ \int_\Om ({\cal D}_1 + {\cal D}_2) dx +M^2+1\right\},
\\
\int_{\Omega}{\cal E}_{2}dx + \int_{\Gamma_{N}}{\cal E}_{3}dS 
\leq C\int_{\Omega}{\cal E}_{2}dx
\leq C\left\{ \int_\Omega (n-p)^2 dx +1\right\}
\leq C\left( \int_\Om {\cal D}_1 dx +M^2+1\right),
\end{gather*}
which lead to the desired inequality.

Let us verify \eqref{ctineq}. {It is seen that
\begin{align*}
\frac{|\nabla n|^2}{n}
 &=n\left| \nabla \log {n} - \nabla v \right|^2 +2\nabla n \cdot \nabla v -n|\nabla v|^2, 
 \\
\frac{|\nabla p|^2}{p}
 &= p\left| \nabla \log {p}+\nabla v \right|^2  -2\nabla p \cdot \nabla v -p|\nabla v|^2.
\end{align*}
These inequalities give
\begin{align}
&\int_{n \ge M} \frac{|\nabla n|^2}{n} dx 
 +\int_{p \ge M} \frac{|\nabla p|^2}{p} dx 
\le \int_\Omega {\cal D}_1 dx 
 +2\int_{n \ge M} \nabla n \cdot \nabla v dx
  -2\int_{p \ge M} \nabla p \cdot \nabla v dx. 
\label{ctineqa}
\end{align}
Owing to the condition $M \ge |N_b|_\infty$,
we can substitute $\phi_3=(n-M)_+$ into \eqref{ddewv} to obtain
\begin{align*}
\int_{n \ge M} \nabla n \cdot \nabla v dx
 &=\int_\Om \nabla (n-M)_+ \cdot \nabla v dx \\
&=-\frac{1}{\ve}\int_{\Gamma_N} (bv-g)(n-M)_+dS
 -\frac{1}{\ve}\int_\Om (n-p-D)(n-M)_+ dx.
\end{align*}
Similarly,
putting $\phi_3=(p-M)_+$ yields
\begin{equation*}
\int_{p \ge M} \nabla p \cdot \nabla v dx
 =-\frac{1}{\ve}\int_{\Gamma_N} (bv-g)(p-M)_+dS
  -\frac{1}{\ve}\int_\Om (n-p-D)(p-M)_+ dx.
\end{equation*}
By substituting these equalities into \eqref{ctineqa},
we have
\begin{align}
\int_{n \ge M} \frac{|\nabla n|^2}{n} dx +\int_{p \ge M} \frac{|\nabla p|^2}{p} dx
& \le \int_\Omega {\cal D}_1 dx 
  +\frac{2}{\ve}\left(\int_{\Gamma_N} I_1 dS 
  +\int_{\Gamma_N} I_2 dS  
  +\int_\Om I_3 dx \right), 
\label{ctineqb} 
\end{align}
where}
\begin{gather*}
I_1:=-bv \{ (n-M)_+ -(p-M)_+\}, 
\quad
I_2:=g\{ (n-M)_+ -(p-M)_+\},
\\
I_3:=-(n-p-D)\{ (n-M)_+ -(p-M)_+\}.
\end{gather*}

Let us first estimate the integral of $I_1$.
From \eqref{ell-Linf},
we see that $|v|_\infty \le C(|n-p|_2+1)$,
and hence
\begin{equation}\label{ctineqb1}
|v|_\infty \le C(J^{1/2}+1).
\end{equation}
By \eqref{tpineq} and the H\"{o}lder inequality,
we have
\begin{align*}
|(n-M)_+|_{1,\Gamma_N}
 &\le C|\nabla (n-M)_+|_1 =C\int_{n \ge M} |\nabla n| dx
\\
&\le C\left( \int_{n \ge M} n dx \right)^{1/2}
 \left( \int_{n \ge M} \frac{|\nabla n|^2}{n} dx \right)^{1/2}
\\
&\le C\left( \int_\Om {\cal I}^M dx \right)^{1/2} J^{1/2}.
\end{align*}
Since we also have the inequality with $n$ replaced by $p$,
we deduce that
\begin{equation}\label{ctineqb2}
|(n-M)_+|_{1,\Gamma_N} +|(p-M)_+|_{1,\Gamma_N}
 \le C \left( \int_\Om {\cal I}^M dx \right)^{1/2} J^{1/2}.
\end{equation}
It follows from \eqref{ctineqb1}, \eqref{ctineqb2} 
and the Schwarz inequality that
\begin{align}
\int_{\Gamma_N} I_1 dS
 &\le |b|_{\infty,\Gamma_N} |v|_\infty |(n-M)_+ -(p-M)_+|_{1,\Gamma_N}
\notag \\
&\le C\left( \int_\Om {\cal I}^M dx \right)^{1/2} (J+J^{1/2})
\notag \\
&\le C\left( \int_\Om {\cal I}^M dx \right)^{1/2} J
 +\mu J +\frac{C}{\mu}\int_\Om {\cal I}^M dx,
\label{I1es}
\end{align}
where $\mu$ is an arbitrary positive number.

Next we deal with $I_2$.
Put $n_M=(n-M)_+ +M$ and $p_M=(p-M)_+ +M$, 
and let $r'<2$ be the H\"{o}lder conjugate of $r$.
Then
\begin{align*}
|n_M|_{r',\Gamma_N} &=|\sqrt{n_M}|_{2r',\Gamma_N}^2
 \le 2|\sqrt{n_M}-\sqrt{M}|_{2r',\Gamma_N}^2 +2|\sqrt{M}|_{2r',\Gamma_N}^2
\\
&\le \mu |\nabla (\sqrt{n_M}-\sqrt{M})|_2^2 +C\mu^{-\al} |\sqrt{n_M}-\sqrt{M}|_1^2 +CM,
\end{align*}
where we have used \eqref{trineq} with $q=2r'$ in deriving the last inequality.
Notice that
\begin{gather*}
|\nabla (\sqrt{n_M}-\sqrt{M})|_2^2 =|\nabla \sqrt{n_M}|_2^2
 =\frac{1}{4}\int_\Om \frac{|\nabla n_M|^2}{n_M} dx
  =\frac{1}{4}\int_{n \ge M} \frac{|\nabla n|^2}{n} dx \le \frac{1}{4}J,
\\
|\sqrt{n_M}-\sqrt{M}|_1^2=\left\{ \int_{n \ge M} (\sqrt{n}-\sqrt{M}) dx\right\}^2
 \le \left( \int_{n \ge M} \sqrt{n}dx\right)^2 \le |\Om| \int_{n \ge M} ndx
  \le |\Om| \int_\Om {\cal I}^M dx.
\end{gather*}
Thus we arrive at
\begin{equation*}
|n_M|_{r',\Gamma_N}
 \le \mu J +C\mu^{-\al} \int_\Om {\cal I}^M dx +CM.
\end{equation*}
Since we also have the inequality with $n$ replaced by $p$,
we see that
\begin{equation*}
|(n-M)_+ -(p-M)_+|_{r',\Gamma_N} =|n_M -p_M|_{r',\Gamma_N}
 \le \mu J +C\mu^{-\al} \int_\Om {\cal I}^M dx +CM.
\end{equation*}
This together with the H\"{o}lder inequality gives
\begin{equation}\label{I2es}
\int_{\Gamma_N} I_2 dS
 \le |g|_{r,\Gamma_N} |(n-M)_+ -(p-M)_+|_{r',\Gamma_N}
  \le C\mu J +C\mu^{-\al} \int_\Om {\cal I}^M dx +CM.
\end{equation}

To estimate $I_3$,
we use the following:
\begin{gather*}
(n-M)_+ -(p-M)_+=n-p -(n-M)_- +(p-M)_-, \\
|(n-M)_- -(p-M)_-| \le M, 
 \qquad |(n-M)_+ -(p-M)_+| \le |n-p|. 
\end{gather*}
From these and the Schwarz inequality,
we deduce that
\begin{align}
I_3 &=-(n-p)^2 +(n-p)\{ (n-M)_- -(p-M)_-\} +D \{ (n-M)_+ -(p-M)_+\}
\notag \\
&\le -(n-p)^2 +M |n-p| +|D| |n-p|
\notag \\
&\le -\frac{1}{2} (n-p)^2 +M^2 +|D|^2.
\label{I3es}
\end{align}
Plugging \eqref{I1es}--\eqref{I3es} into \eqref{ctineqb} 
and choosing $\mu$ appropriately small, 
we obtain \eqref{ctineq}.
Thus the proof is complete.
\end{proof}

We are now in a position to prove Proposition~\ref{asulob}.
\begin{proof}[Proof of Proposition~\ref{asulob}]
Put
\begin{equation*}
E=E(t):=\int_\Om ({\cal E}_1+{\cal E}_2) dx +\int_{\Gamma_N} {\cal E}_3 dS
\end{equation*}
and define
\begin{gather*}
\theta_0:=\min \left\{ \theta_1, \frac{1}{4C_{1} \tilde C_1 \max\{\delta,1\}},
 \frac{c_1}{4\tilde C_1 (5C_1 \max \{ \de,1\}+1)} \right\},
\\
M_2:=\max \left\{ h^{-1}\left( \frac{1}{\theta_0}\right), M_1,
 \max \{ |N_b|_\infty, |P_b|_\infty \}\right\},
\end{gather*}
where $\theta_{1}$, $C_{1}$, $\tilde C_{1},c_1$ and $M_1$ are the constants 
given in Lemmas~\ref{EnergyEq}--\ref{lemefs}.

In what follows, we assume that $\theta<\theta_0$.
The proof is completed by showing that
\begin{equation}\label{eloineq}
E' + \tilde{c} E \le \tilde{C}
\end{equation}
in some interval $(T_0,\infty)$,
where $\tilde{c}$ and $\tilde{C}$ are positive constants
depending only on $\max\{\de,1\}$, $c_0$, $C_0$, $\Om$, $\Gamma_D$, $\rho$, $r$, $\ve$ and $c_R$.
Indeed, applying the Gronwall inequality gives
\begin{equation*}
\limsup_{t \to \infty} E(t) \le \limsup_{t \to \infty}
 \left\{ E(T_0) e^{-\tilde{c} (t-T_0)}+\frac{\tilde{C}}{\tilde{c}} \left( 1-e^{-\tilde{c} (t-T_0)}\right) \right\}
  =\frac{\tilde{C}}{\tilde{c}}.
\end{equation*}

First let us determine $T_0$. 
By the definitions of $\theta_0$ and $M_2$, we have
\begin{equation*}
\tilde C_{1}\left(\theta_0 +\frac{1}{h(M_2)}\right)
\leq 2\tilde C_{1}\theta_0
\leq \frac{1}{2C_{1} \max\{\delta,1\}}.
\end{equation*}
Hence we see from \eqref{emineq2} that
\begin{align*}
\int_\Omega (n+p) dx
\le  \int_\Omega {\cal I}^{M_2} dx+ M_2 |\Omega| 
\le \frac{1}{2C_1 \max\{\delta,1\}} \left\{ \int_\Omega ({\cal D}_1+{\cal D}_2) dx +1\right\}
 +  M_2 |\Omega|.
\end{align*}
This together with \eqref{emineq0} gives 
\begin{equation}
E' +\frac{1}{4}\int_\Om ({\cal D}_1+{\cal D}_2) dx \le C_2 \de,
\label{emineq}
\end{equation} 
where $C_2=C_2(\max\{\de,1\},c_0,C_0,\Omega,\Gamma_D,\rho,r,\varepsilon,c_R)>0$ 
is a constant.
Integrating this over $[t_0,t] \subset I$ leads to
\begin{equation*}
\limsup_{t \to \infty} \frac{1}{t-t_0} \int_{t_0}^t \int_\Om ({\cal D}_1+{\cal D}_2) dx dt 
 \le \limsup_{t \to \infty} \left( \frac{4E(t_0)}{t-t_0}+4C_2 \de \right) =4C_2 \de.
\end{equation*}
Therefore we can take $T_0>0$ such that
\begin{equation}\label{dtzes}
\left. \int_\Om ({\cal D}_1+{\cal D}_2) dx\right|_{t=T_0} \le 5C_2 \de.
\end{equation}

The constants $\tilde{c}$ and $\tilde{C}$ are chosen as follows.
The definitions of $\theta_0$ and $M_2$ give
\begin{equation}
\tilde C_1 \left( \theta_0+\frac{1}{h(M_2)}\right) \leq \frac{c_1}{2 (5C_1 \max \{ \de,1\}+1)}.
\label{hmti}
\end{equation}
This together with \eqref{emineq2} and \eqref{dtzes} shows that
\begin{equation*}
\left. \int_\Om {\cal I}^{M_2} dx\right|_{t=T_0}
 \le \frac{c_1}{2 (5C_1 \max \{ \de,1\}+1)}
  \left\{ \left. \int_\Om ({\cal D}_1+{\cal D}_2) dx\right|_{t=T_0} +1\right\} \le \frac{c_1}{2}.
\end{equation*}
This 
gives $\int_\Om {\cal I}^{M_2} dx \le c_1$ in some open interval $J_0 \ni T_0$.
Applying Lemma~\ref{lemefs}, we deduce that
\begin{equation}\label{eobbd}
E  \le C\left\{ \int_\Omega ({\cal D}_1 +{\cal D}_2) dx +1\right\}
  \quad \textrm{in } J_0,
\end{equation}
where $C=C(\max\{ \de,1\},c_0,C_0,\Om,\Gamma_D,\rho,r,\ve,c_R)>0$ is a constant.
This together with \eqref{emineq} yields
\begin{equation}\label{eloineq0}
E' +c_3 E \le C_3 \quad \textrm{in } J_0.
\end{equation}
Furthermore, \eqref{emineq2}, \eqref{dtzes} and \eqref{eobbd} imply that
\begin{equation}\label{etzub}
E(T_0) \le C_4.
\end{equation}
Here $c_3$, $C_3$ and $C_4$ are positive constants depending only on 
$\max\{ \de,1\},c_0,C_0,\Om,\Gamma_D,\rho,r,\ve$ and $c_R$.
Now we choose $\tilde{c}$ and $\tilde{C}$ as 
\begin{equation*}
\tilde{c}:=\min \left\{ c_3, \frac{C_2 \max \{ \de,1\}}{4C_4} \right\},
\quad
\tilde{C}:=C_3.
\end{equation*}

We complete the proof by showing \eqref{eloineq}.
From \eqref{eloineq0}, 
we see that the set 
\begin{equation*}
{\cal T}:= \{ \tau \in (T_0,\infty) ; \eqref{eloineq} \ \mbox{holds in} \
 (T_0,\tau) \}
\end{equation*}
is nonempty,
and therefore $T_1:=\sup {\cal T} \in (T_0,\infty]$.
What is left is to show that $T_1=\infty$.
On the contrary, suppose that $T_1<\infty$.
Then we have either 
\begin{equation*}
\left. \int_\Om {\cal I}^{M_2} dx\right|_{t=T_1} <c_1
\quad \mbox{or} \quad \left. \int_\Om {\cal I}^{M_2} dx\right|_{t=T_1} \ge c_1.
\end{equation*}
We first consider the former case.
In the same way as the derivation of \eqref{eloineq0}, 
one can show that \eqref{eloineq} holds in some open interval $J_1 \ni T_1$.
From this we find that $J_1 \subset {\cal T}$,
which contradicts the fact that $T_1$ is the supremum of ${\cal T}$.
Next let us consider the latter case.
We take $T_2>T_1$ such that $\int_\Om {\cal I}^{M_2} dx \ge c_1/2$ on $[T_1,T_2]$.
Then, from \eqref{emineq2} and \eqref{hmti}, we have
\begin{equation*}
\int_\Om ({\cal D}_1+{\cal D}_2) dx \ge \frac{2 (5C_2 \max \{ \de,1\}+1)}{c_1} \int_\Om {\cal I}^{M_2} dx-1
\geq 5C_2 \max \{ \de,1\}
\end{equation*}
on $[T_1,T_2]$.
Plugging this into \eqref{emineq} leads to
\begin{equation}\label{elces}
E' \le -\frac{1}{4}C_2 \max \{ \de,1\}
 \quad \textrm{on } [T_1,T_2],
\end{equation}
which particularly yields
\begin{equation}\label{etoub}
E \le E(T_1) \quad \textrm{on } [T_1,T_2].
\end{equation}
Since \eqref{eloineq} holds on $[T_0,T_1]$,
we see from the Gronwall inequality that 
\begin{equation*}
E(T_1) \le E(T_0) e^{-\tilde{c} (T_1-T_0)}+\frac{\tilde{C}}{\tilde{c}} \left( 1-e^{-\tilde{c} (T_1-T_0)}\right)
 \le E(T_0)+\frac{\tilde{C}}{\tilde{c}}.
\end{equation*}
This together with \eqref{etzub} and \eqref{etoub} shows that
\begin{equation}\label{euboi}
E \le C_4+\frac{\tilde{C}}{\tilde{c}} \quad \textrm{on } [T_1,T_2].
\end{equation}
By \eqref{elces}, \eqref{euboi} and the definition of $\tilde{c}$,
we obtain 
\begin{equation*}
E' +\tilde{c} E \le -\frac{1}{4}C_2 \max \{ \de,1\} +\tilde{c} C_4 +\tilde{C} \le \tilde{C}
 \quad \textrm{on } [T_1,T_2].
\end{equation*}
This gives $T_2\in {\cal T}$, a contradiction.
We thus conclude that $T_1=\infty$, and the proof is complete.
\end{proof}

\subsection{Proof of Proposition~\ref{libblb}}\lb{S4.2}

We prove Proposition~\ref{libblb} by the iteration argument of Moser.
To this end, we put
\begin{align*}
J_\ga&=J_\ga(t):=\int_\Om \left\{ (n-M_0)_+^{\ga}+ (p-M_0)_+^{\ga} \right\} dx, 
\\
K_\ga&=K_\ga(t):=\int_\Om \left\{ \left| \nabla (n-M_0)_+^{\ga} \right|^2
 +\left| \nabla (p-M_0)_+^{\ga} \right|^2 \right\} dx,
\end{align*}
where  $M_0:=\max \{ |N_b|_\infty, |P_b|_\infty \}$ and $\ga \ge 1$.

\begin{lem}\lb{pselem}
There exist constants $C=C(C_0,\Omega,\Gamma_D,r,\varepsilon,c_R)>0$
and $\be=\be(r)>0$ such that
\begin{equation}\lb{pseineq}
J_{2\ga}' +K_{\ga}
 \le C(\|v\|_1+1)^\be \ga^\be (J_\ga^2+1).
\end{equation}
\end{lem}
\begin{proof}
We put $\zeta:=(n-M_0)_+^\ga$ and $\xi:=(p-M_0)_+^\ga$.
Take $\phi_1=(n-M_0)_+^{2\ga-1}$ in \eqref{ddewn} to obtain
\begin{align}
&\langle n',(n-M_0)_+^{2\ga-1} \rangle
 +(2\ga-1) \int_\Om (n-M_0)_+^{2\ga-2} \nabla n \cdot \nabla (n-M_0)_+ dx
\notag\\
&\quad -\int_\Om n \nabla v \cdot \nabla (n-M_0)_+^{2\ga-1} dx
 +\int_\Om R(n,p)(n-M_0)_+^{2\ga-1} dx=0.
\lb{M1}
\end{align}
The first two terms of the left-hand side of this equality are written as
\begin{gather}
\langle n',(n-M_0)_+^{2\ga-1} \rangle
 =\frac{1}{2\ga}\frac{d}{dt} \int_\Om (n-M_0)_+^{2\ga} dx
  =\frac{1}{2\ga}\frac{d}{dt} |\zeta|_2^2,
\lb{M2}\\
\int_\Om (n-M_0)_+^{2\ga-2} \nabla n \cdot \nabla (n-M_0)_+ dx
 =\frac{1}{\ga^2} \int_\Om \left| \nabla (n-M_0)_+^\ga \right|^2 dx
  =\frac{1}{\ga^2} |\nabla \zeta|_2^2.
\lb{M3}
\end{gather}
Here \eqref{M2} is validated by \eqref{reg0}.
Note that 
$n \nabla (n-M_0)_+^{2\ga-1}=\nabla F(n)$,
where
\begin{equation*}
F(n):=\frac{2\ga-1}{2\ga} (n-M_0)_+^{2\ga} +M_0(n-M_0)_+^{2\ga-1}.
\end{equation*}
Hence, using \eqref{ddewv} with $\phi_3=F(n)$,
we have
\begin{align}
\int_\Om n \nabla v \cdot \nabla (n-M_0)_+^{2\ga-1} dx
 &=\int_\Om \nabla v \cdot \nabla F(n) dx
\notag\\
&= -\frac{1}{\ve} \int_{\Gamma_N} (bv-g) F(n) dS
-\frac{1}{\ve} \int_\Om (n-p-D) F(n) dx.
\lb{M4}
\end{align}
Substituting \eqref{M2}--\eqref{M4} into \eqref{M1} yields
\begin{gather}
\frac{1}{2\ga} \frac{d}{dt} |\zeta|_2^2 +\frac{2\ga-1}{\ga^2} |\nabla \zeta|_2^2
 +\frac{1}{\ve} \int_\Om (n-p) F(n) dx =  \frac{1}{\ve} \int_{\Gamma_N} I_1 dS
 +\int_\Om I_2 dx,
\lb{mifeq} \\
I_1:=-(bv-g) F(n),
\quad
I_2:=\frac{1}{\ve}D F(n)-R(n,p) \zeta^{2-1/\ga}.
\nonumber
\end{gather}

Let us estimate the right-hand side of \eqref{mifeq}.
From now on, 
let $C$ denote a positive constant depending only on 
$C_0,\Omega,\Gamma_D,r,\varepsilon$ and $c_R$.
By the Young inequality,
we have
\begin{align}\lb{M5}
F(n)=\frac{2\ga-1}{2\ga}\zeta^2 +M_0\zeta^{2-1/\ga}
 \le \frac{2\ga-1}{2\ga}\zeta^2 +M_0\left( \frac{2\ga-1}{2\ga}\zeta^2 +\frac{1}{2\ga}\right)
  \le C(\zeta^2 +1).
\end{align}
From this, the H\"{o}lder inequality, \eqref{tsineq} and \eqref{trineq},
we see that
\begin{align}
\int_{\Gamma_N} I_1 dS
 &\le C\int_{\Gamma_N} (b|v|+|g|) (\zeta^2 +1) dS
\nonumber \\
&\le C|b|_{\infty,\Gamma_N} |v|_{4,\Gamma_N} |\zeta^2 +1|_{4/3,\Gamma_N}
 +C|g|_{r,\Gamma_N} |\zeta^2 +1|_{r',\Gamma_N}
\nonumber \\
&\le C\left( |v|_{4,\Gamma_N}+1\right) \left( |\zeta|_{8/3,\Gamma_N}^2 +|\zeta|_{2r',\Gamma_N}^2 +1\right)
\nonumber \\
&\le C\left( \| v\|_1+1\right) \left( \mu |\nabla \zeta|_2^2 +\mu^{-\tilde \al} |\zeta|_1^2+1\right),
\lb{miites}
\end{align}
where $r'<2$ is the H\"{o}lder conjugate of $r$,
$\tilde \al=\tilde \al (r)>0$ is a constant, and $\mu>0$ is an arbitrary number.
Furthermore, \eqref{M5} and $|R(n,p)| \le c_R (n+p+1)$ together with the Young inequality yield
\begin{align*}
I_2 
 \le C(\zeta^2 +\xi^2+1).
\end{align*}
By the Galiardo-Nirenberg, Poincar\'{e} and Young inequalities,
we have 
\begin{align}\lb{miites1}
\int_\Om I_2 dx \le & C(|\nabla \zeta|_2^{6/5} |\zeta|_1^{4/5}+|\nabla \xi|_2^{6/5} |\xi|_1^{4/5})+C
\notag \\
\leq &C\mu (|\nabla \zeta|_2^2+|\nabla \xi|_2^2) 
+C\mu^{-3/2} (|\zeta|_1^2+|\xi|_1^2) +C.
\end{align}

Plugging \eqref{miites} and \eqref{miites1} into \eqref{mifeq},
we deduce that
\begin{align*}
&\frac{d}{dt} |\zeta|_2^2 +\frac{2(2\ga-1)}{\ga} |\nabla \zeta|_2^2
 +\frac{2\ga}{\ve} \int_\Om (n-p) F(n) dx 
\\
&\le C\left( \| v\|_1+1\right) \ga 
 \left\{ \mu (|\nabla \zeta|_2^2+|\nabla \xi|_2^2) 
  +\max \{ \mu^{-\tilde \al},\mu^{-3/2}\} (|\zeta|_1^2+|\xi|_1^2) +1\right\}.
\end{align*}
Performing the same computation for $\xi$
and adding the result to the above inequality,
we obtain
\begin{align*}
&\frac{d}{dt} J_{2\ga} +\frac{2(2\ga-1)}{\ga} K_\ga
 +\frac{2\ga}{\ve} \int_\Om (n-p) (F(n)-F(p)) dx 
\\
&\le C\left( \| v\|_1+1\right) \ga 
 \left( \mu K_\ga +\max \{ \mu^{-\tilde \al},\mu^{-3/2}\} J_\ga^2 +1\right).
\end{align*}
Note that $(n-p) (F(n)-F(p)) \ge 0$,
since $F(z)$ is nondecreasing in $z$.
Therefore, by choosing $\mu$ as $\mu=c\left( \| v\|_1+1\right)^{-1}\ga^{-1}$
with a suitable constant $c=c(C_0,\Omega,\Gamma_D,r,\varepsilon,c_R)>0$,
we conclude that \eqref{pseineq} holds with $\beta:=\max \{ \tilde \al+1, 5/2\}$.
\end{proof}

Let us prove Proposition~\ref{libblb}.

\begin{proof}[Proof of Proposition~\ref{libblb}]
In the proof, $c$ and $C$ denote positive constants
depending only on $L$, $c_0$ $C_0$, $\Omega$, $\Gamma_D$, $r$, $\varepsilon$ and $c_R$.
From the definition of $L$,
we can take $\tau_0 \in I$ such that for all $t \ge \tau_0$,
\begin{equation}\lb{joves}
J_1(t) +\| v(t)\|_1 \le 2L.
\end{equation}

We first take $\ga=1$ in \eqref{pseineq}.
Since the Poincar\'{e} inequality gives $K_1 \geq cJ_2$,
we see from the Gronwall inequality that
\begin{equation*}
J_2(t) \le J_2(\tau_0) e^{-c(t-\tau_0)}
 +C\int_{\tau_0}^t e^{-c(t-\tau)} (\|v(\tau)\|_1+1)^\be (J_1(\tau)^2+1) d\tau
\end{equation*}
for all $t \ge \tau_0$.
From this and \eqref{joves}, 
we have
\begin{equation*}
\limsup_{t \to \infty} J_2(t)
 \le \limsup_{t \to \infty} \left\{ J_2(\tau_0) e^{-c(t-\tau_0)}
  +C\int_{\tau_0}^t e^{-c(t-\tau)}d\tau \right\} = \frac{C}{c}.
\end{equation*}
This particularly gives
\begin{equation}\lb{jtes}
\limsup_{T \to \infty} \int_{T-1}^{T+1} \int_\Om (n^2 +p^2) dxdt
 \le C\limsup_{T \to \infty} \int_{T-1}^{T+1} (J_2(t) +1) dt \le C.
\end{equation}
By the iteration argument of Moser~\cite{M64},
one can show that for all $T \ge \tau_0+1$,
\begin{equation}\lb{whineq}
\sup_{[T,T+1] \times \Om} (n+p)
 \le C\left\{ \int_{T-1}^{T+1} \int_\Om (n^2 +p^2) dxdt \right\}^{1/2} +C.
\end{equation}
The proposition immediately follows by combining \eqref{jtes} and \eqref{whineq}.

To complete the proof, we briefly derive \eqref{whineq}.
Let $0<\kappa \le 1/2$ and let $\rho \in C^\infty(\bbr)$ satisfy
\begin{equation}\lb{prrh}
\rho(t)=\left\{
\begin{aligned}
&0 && \text{for } t \le T-2\kappa, 
\\
&1 && \text{for } t \ge T-\kappa,
\end{aligned}
\right.
\qquad
0 \le \rho'(t) \le \frac{2}{\kappa}. 
\end{equation}
Multiplying \eqref{pseineq} by $\rho$ and integrating it, we
see that for all $T-1 \le t_1 \le T+1$,
\begin{align}
J_{2\ga}(t_1)\rho(t_1) +\int_{T-1}^{t_1} K_\ga \rho dt
 &\le \int_{T-1}^{t_1} J_{2\ga} \rho' dt +C\ga^\be \int_{t-1}^{t_1} (J_\ga^2+1) \rho dt
\notag \\
 &\le C\left( \frac{1}{\kappa} +\ga^\be \right) \left(\int_{T-2\kappa}^{T+1} J_{2\ga} dt+1\right),
\lb{D1}
\end{align}
where we have used  \eqref{joves} and \eqref{prrh} and the fact that $J_\ga^2 \le CJ_{2\ga}$
in deriving the last inequality.
We take $t_1=T+1$ in \eqref{D1} to obtain
\begin{equation}
\int_{T-\kappa}^{T+1} K_\ga dt
 \le C\left( \frac{1}{\kappa} +\ga^\be \right) 
\left(\int_{T-2\kappa}^{T+1} J_{2\ga} dt+1\right).
\lb{D2}
\end{equation}
We choose $t_1 \in [T-\kappa,T+1]$ such that
\begin{equation*}
J_{2\ga}(t_1)=\max_{t \in [T-\kappa,T+1]} J_{2\ga}(t).
\end{equation*}
Then \eqref{D1} also gives
\begin{equation}
\max_{t \in [T-\kappa,T+1]} J_{2\ga}(t)
 \le C\left( \frac{1}{\kappa} +\ga^\be \right) 
\left(\int_{T-2\kappa}^{T+1} J_{2\ga} dt+1\right).
\lb{D3}
\end{equation}
We know from \cite[Lemma~2]{M64} that for $\la=5/3$,
\begin{equation*}
\int_{T-\kappa}^{T+1} J_{2\la \ga} dt
 \le C\left( \max_{t \in [T-\kappa,T+1]} J_{2\ga}(t) \right)^{2/3}
  \int_{T-\kappa}^{T+1} K_{\ga} dt.
\end{equation*}
This together with \er{D2} and \er{D3} leads to 
\begin{equation*}
\left( \int_{T-\kappa}^{T+1} J_{2\la \ga} dt \right)^{1/(2\la \ga)}+1
 \le \tilde{C}^{1/(2\ga)} \left( \frac{1}{\kappa} +\ga^\be \right)^{1/(2\ga)}
  \left\{ \left( \int_{T-2\kappa}^{T+1} J_{2\ga} dt \right)^{1/(2\ga)} +1\right\},
\end{equation*}
where $\tilde{C}=\tilde C(L,C_0,\Omega,\Gamma_D,r,\varepsilon,c_R)>0$.
Substituting $\ga=\ga_n:=\la^n$ and $\kappa=\kappa_n:=2^{-n-1}$ into this inequality, we have
$I_{n+1} \le \Lambda_n^{1/(2\ga_n)} I_n$, where
\begin{equation*}
I_n:=\left( \int_{T-2\kappa_n}^{T+1} J_{2\ga_n} dt \right)^{1/(2\ga_n)}+1,
\quad
\Lambda_n:=\tilde C( 2^{n+1}+\la^{\be n}),
\quad
n=0,1,\ldots.
\end{equation*}
Hence we see that
\begin{equation*}
I_n \le \prod_{k=0}^n \Lambda_k^{1/(2\ga_k)} I_0 \leq C I_0.
\end{equation*}
Letting $n \to \infty$ gives \eqref{whineq}, and the proof is complete.
\end{proof}

\subsection{Proof of Proposition~\ref{asubfb}}\lb{S4.3}

To obtain the lower bounds of $n$ and $p$,
we derive an inequality similar to \eqref{pseineq}.
Let $m_0$ be a constant defined by
\begin{equation*}
m_0:=\min \left\{ c_0, \frac{\ve c_R}{(2c_0+2+|D|_\infty+\ve c_R)^2} \right\}.
\end{equation*}
For $\ga>0$, $\ga \neq 1$, $0<\sigma \le m_0/2$ and $z \ge 0$, we write
\begin{gather*}
G_0 (z):=\left\{ (z_\sigma)^{-\ga} -m_0^{-\ga} \right\}_+
 =(z_\sigma^{m_0})^{-\ga} -m_0^{-\ga},
\quad
G_1 (z):=\int_z^{m_0} G_0(y) dy,
\\
G_2 (z):=\int_z^{m_0} \sqrt{-\frac{dG_0}{dz}(y)} dy
 =\frac{2\sqrt{\ga} }{\ga-1}
  \left\{ (z_\sigma^{m_0})^{\frac{-\ga+1}{2}} -m_0^{\frac{-\ga+1}{2}} \right\},
\\
G_3 (z):=\int_z^{m_0} -y\frac{dG_0}{dz}(y) dy
 =\frac{\ga }{\ga-1}
  \left\{ (z_\sigma^{m_0})^{-\ga+1} -m_0^{-\ga+1} \right\},
\end{gather*}
where 
\begin{equation*}
z_\sigma^{m_0}:=(z_\sigma-m_0)_+ +m_0
=\left\{
\begin{aligned}
&\sigma && (z \le \sigma),
\\
&z && (\sigma<z<m_0),
\\
&m_0 && (z \ge m_0).
\end{aligned}
\right.
\end{equation*}

\begin{lem}
There exist constants $C=C(C_0,\Omega,\Gamma_D,r,\varepsilon)>0$
and $\be=\be(r)>0$ such that
\begin{align}
&\frac{d}{dt} \int_\Om G_1 (n) dx
 +\frac{1}{2} \int_\Om |\nabla G_2(n)|^2 dx
  +\sqrt{\frac{c_R}{\ve}} \int_\Om \sqrt{G_3(n)G_0(n)}  dx 
\nonumber
\\
&\le C\left( \| v\|_1+1\right)^\be |\ga-1|^\be \int_\Om G_2(n)^2 dx
 +\frac{C\left( \| v\|_1+1\right)\ga}{|\ga-1|}m_0^{-2\ga+2}.
\label{psaineq}
\end{align}
\end{lem}
\begin{proof}
We can choose $\phi_1=G_0(n) \in H^1_D (\Om)$ as a test function in \eqref{ddewn}.
Then we have
\begin{equation*}
\frac{d}{dt} \int_\Om G_1 (n) dx
 +\int_\Om |\nabla G_2(n)|^2 dx
  +\int_\Om n \nabla v \cdot \nabla G_0(n) dx
   -\int_\Om R(n,p)G_0(n) dx=0,
\end{equation*}
where we have used the fact that
\begin{equation*}
\langle n',G_0(n) \rangle
 =-\frac{d}{dt} \int_\Om G_1(n) dx,
\quad
\int_\Om \nabla n \cdot \nabla G_0(n) dx
 =-\int_\Om |\nabla G_2(n)|^2 dx.
\end{equation*}
By using \eqref{ddewv} with $\phi_3=G_3(n) \in H^1_D (\Om)$,
the third term of the left-hand side is computed as
\begin{align*}
\int_\Om n\nabla v \cdot \nabla G_0(n) dx
 =\int_\Om \nabla v \cdot \nabla G_3(n) dx
  = -\frac{1}{\ve} \int_{\Gamma_N} (bv-g) G_3(n) dS
   -\frac{1}{\ve} \int_\Om (n-p-D) G_3(n) dx.
\end{align*}
Therefore we arrive at
\begin{gather}
\frac{d}{dt} \int_\Om G_1 (n) dx +\int_\Om |\nabla G_2(n)|^2 dx
 +\int_\Om I_2 dx =\frac{1}{\ve} \int_{\Gamma_N} I_1 dS,
\lb{mifeq2} \\
I_1:=(bv-g) G_3(n),
\quad
I_2:=-\frac{1}{\ve}(n-p-D)G_3(n) -R(n,p)G_0(n).
\nonumber
\end{gather}

Notice that
$G_3(z)=(\ga-1)G_2(z)^2/4 +\sqrt{\ga} m_0^{-\ga+1}G_2(z)$.
From this, we have
\begin{equation*}
G_3(z)
 \le |\ga-1|G_2(z)^2 +\frac{\ga}{|\ga-1|}m_0^{-2\ga+2}.
\end{equation*}
Therefore, in the same way as the derivation of \eqref{miites},
the right-hand side of \eqref{mifeq2} is estimated as
\begin{align}
\int_{\Gamma_N} I_1 dS &\le \int_{\Gamma_N} (b|v|+|g|)
 \left( |\ga-1|G_2(n)^2 +\frac{\ga}{|\ga-1|}m_0^{-2\ga+2}\right) dS
\nonumber
\\
&\le \frac{\ve}{2} |\nabla G_2(n)|_2^2 
 +C\left( \| v\|_1+1\right)^\be |\ga-1|^\be |G_2(n)|_1^2
  +\frac{C\ga}{|\ga-1|}\left( \| v\|_1+1\right) m_0^{-2\ga+2}.
\label{psa1}
\end{align}
Here $\be=\be (r)>0$ and $C=C(C_0,\Omega,\Gamma_D,r,\varepsilon)>0$ are constants.
To estimate $I_2$,
we note that
\begin{equation}
z G_0(z)=\int_z^{m_0} -z\frac{dG_0}{dz}(y) dy \le G_3(z)
 \le \int_z^{m_0} -m_0\frac{dG_0}{dz}(y) dy=m_0 G_0(z).
\label{G0G3ineq}
\end{equation}
Hence we see that
\begin{align*}
I_2&=\left\{ \frac{1}{\ve}(n+p+2)G_3(n)+\frac{c_R}{n+p+2}G_0(n) \right\}
 -\left\{ \frac{1}{\ve}(2n+2-D)G_3(n) +\frac{c_R p}{n+p+2} nG_0(n) \right\}
\\
&\ge 2\sqrt{\frac{c_R}{\ve}G_3(n)G_0(n)}
 -\left\{ \frac{1}{\ve}(2m_0+2+|D|_\infty)G_3(n) +c_R G_3(n) \right\}
\\
&\ge 2\sqrt{\frac{c_R}{\ve}G_3(n)G_0(n)}
 -\left\{ \frac{1}{\ve}(2m_0+2+|D|_\infty) +c_R \right\} \sqrt{m_0 G_3(n)G_0(n)}.
\end{align*}
Here, in deriving the first inequality, 
we have used the Schwarz inequality, the fact that $G_3(n)=0$ for $n \le m_0$
and \eqref{G0G3ineq},
and in deriving the second inequality, we have used \eqref{G0G3ineq}.
By the definition of $m_0$,
we conclude that
\begin{equation}\label{psa2}
I_2 \ge \sqrt{\frac{c_R}{\ve}G_3(n)G_0(n)}.
\end{equation}
Plugging \eqref{psa1} and \eqref{psa2} into \eqref{mifeq2} yields the desired inequality.
\end{proof}

We conclude this section by proving Proposition~\ref{asubfb}.

\begin{proof}[Proof of Proposition~\ref{asubfb}]
In the proof, we denote by $C$ a positive constant depending only on 
$\tilde L$, $c_0$, $C_0$, $\Omega$, $\Gamma_D$, $r$, $\varepsilon$ and $c_R$.
By the assumption.
we can choose $\tau_0>0$ such that $\| v(t)\|_1 \le 2 \tilde L$ for all $t \ge \tau_0$.

In what follows, we fix $T \ge \tau_0+1$.
Let us show that
\begin{equation}\lb{nmgi} 
\int_{T-1}^{T+1} \int_\Om (n_\sigma^{m_0})^{-\frac{1}{4}} dxdt \le C.
\end{equation}
For this purpose,
we integrate both sides of \eqref{psaineq} over $[T-1,T+1]$ to obtain
\begin{align}
&\int_{T-1}^{T+1} \int_\Om \sqrt{G_3(n)G_0(n)} dxdt
\nonumber
\\
&\le \left. \int_\Om G_1 (n) dx \right|_{t=T-1}
 +C\left( |\ga-1|^\be \int_{T-1}^{T+1} \int_\Om G_2(n)^2 dxdt
  +\frac{\ga}{|\ga-1|}m_0^{-2\ga+2} \right).
\lb{nmgifs}
\end{align}
It is elementary to show that if $\ga<1$, then
\begin{gather*}
G_1(z) 
 \le \frac{1}{1-\ga}m_0^{1-\ga},
\quad
G_2(z) \le \frac{2\sqrt{\ga}}{1-\ga} m_0^{\frac{1-\ga}{2}},
\quad
G_3(z) \ge \frac{\ga}{1-\ga} \left( 1-\frac{1}{2^{1-\ga}}\right) m_0^{1-\ga}
 \chi_{\{ z \le m_0/2\}}.
\end{gather*}
Hence, by choosing $\ga=1/2$ in \eqref{nmgifs},
we obtain \eqref{nmgi}.

From now on, 
we suppose that $\ga>1$.
Let $t_1 \in [T-1,T+1]$ and 
let $\rho \in C^\infty (\bbr)$ be a nonnegative function satisfying $\rho (T-1)=0$.
Then, multiplying \eqref{psaineq} by $\rho$ and integrating over $[T-1,t_1]$ yield
\begin{align*}
&\left. \rho \int_\Om G_1(n) dx \right|_{t=t_1}
 +\int_{T-1}^{t_1} \int_\Om |\nabla G_2(n)|^2 \rho dxdt
\\
&\le \int_{T-1}^{t_1} \int_\Om G_1(n) \rho' dxdt
 +C\left\{ (\ga-1)^\be \int_{T-1}^{t_1} \int_\Om G_2(n)^2 \rho dxdt
   +\frac{\ga}{\ga-1} m_0^{-2\ga+2} \int_{T-1}^{t_1} \rho dt \right\}.
\end{align*}
Note that
\begin{gather*}
\frac{1}{\ga-1}(z_\sigma^{m_0})^{-\ga+1}-\frac{\ga}{\ga-1}m_0^{-\ga+1}
 \le G_1(z) \le \frac{\ga}{\ga-1}(z_\sigma^{m_0})^{-\ga+1},
\qquad
G_2(z) 
 \le \frac{2\sqrt{\ga}}{\ga-1} (z_\sigma^{m_0})^{-\frac{\ga-1}{2}}.
\end{gather*}
From these, 
we have
\begin{align*}
&\left. \rho \int_\Om (n_\sigma^{m_0})^{-(\ga-1)} dx \right|_{t=t_1}
 +\frac{\ga}{\ga-1} \int_{T-1}^{t_1} \int_\Om 
  \left| \nabla \left\{ (n_\sigma^{m_0})^{\frac{-\ga+1}{2}} \right\}\right|^2 \rho dxdt
\\
&\le C\int_{T-1}^{t_1} \int_\Om 
 (n_\sigma^{m_0})^{-(\ga-1)} \left( \ga |\rho'| +\ga^\be \rho \right) dxdt
  +C\ga m_0^{-2\ga+2} \int_{T-1}^{t_1} \rho dt +C\ga m_0^{-\ga+1}.
\end{align*}
Thus, by the same argument as in the derivation of \eqref{whineq}, 
we deduce that
\begin{equation*}
\sup_{[T,T+1] \times \Om} (n_\sigma^{m_0})^{-1}
 \le C\left\{ \int_{T-1}^{T+1} \int_\Om (n_\sigma^{m_0})^{-\frac{1}{4}} dxdt \right\}^4 +C.
\end{equation*}
Plugging \eqref{nmgi} into this inequality and letting $\sigma \to 0$,
we obtain $|n(T)^{-1}|_\infty \le C$.
The inequality $|p(T)^{-1}|_\infty \le C$ can be shown in the same way,
and therefore the proof is complete.
\end{proof}

\section{Estimates of the difference of solutions}\lb{energy}
In this section we estimate the relative entropy of any two solutions 
$(n_1,p_1,v_1)$ and $(n_2,p_2,v_2)$ of \er{ddeq}.
Theorem~\ref{uesthm} ensures that if $\theta<\theta_0$, 
then $(n_1,p_1,v_1)$ and $(n_2,p_2,v_2)$ satisfy
\begin{equation}\lb{asp1}
(2\hat C)^{-1} \le n_1,n_2,p_1,p_2 \le 2\hat C \quad \text{in } (\tilde t,\infty) \times \Omega
\end{equation}
for some $\tilde t \in \mathbb R$, where
$\hat C=\hat C(c_0,C_0,\Omega,\Gamma_D,\rho,r,\ve,c_R)$ 
is the constant being in \eqref{gsulb}.
Throughout this section, we suppose that $\theta<\theta_0$ and $t \geq \tilde{t}$.
We set 
\begin{equation*}
\varphi:=\frac{n_1}{n_2}-1, \quad \psi:=\frac{p_1}{p_2}-1, \quad \eta:=v_1-v_2.
\end{equation*}
The goal of this section is to prove the following proposition.
\begin{pro}\lb{abDiff}
There exist positive constants $\de_0$, $c$ and $C$ depending only on 
$c_0$, $C_0$, $\Om$, $\Gamma_D$, $\rho$, $r$, $\ve$ and $c_R$ such that
if $\de<\de_0$, then the following inequalities hold:
\begin{gather}
|n_2 \varphi (t)|_2+|p_2 \psi (t)|_2+\| \eta (t)\|_1 \le Ce^{-c(t-\tilde t)},
\lb{dgsineq} \\
\int_{\tilde t}^t e^{c(s-\tilde t)}
 \left( |\nabla \log (1+\varphi) (s)|_2^2+|\nabla \log (1+\psi) (s)|_2^2 \right) ds \le C.
\lb{ddgsineq}
\end{gather}
\end{pro}

For the proof, let us first find the equations for $\varphi$, $\psi$ and $\eta$.
By \eqref{ddewn},
we have
\begin{equation*}
\langle n_1'-n_2', \phi_1 \rangle
+\int_\Om \left\{ (\nabla n_1 -n_1 \nabla v_1-\nabla n_2 +n_2 \nabla v_2) \cdot \nabla \phi_1
+(R(n_1,p_1)-R(n_2,p_2)) \phi_1 \right\}dx=0,
\end{equation*}
where $\phi_1 \in H^1_D(\Om)$. 
From the following two equalities
\begin{gather*}
n_1-n_2=n_2 \varphi,
\\
\nabla n_1 -n_1 \nabla v_1-\nabla n_2 +n_2 \nabla v_2 
=n_1(\nabla \log(1+\varphi)-\nabla \eta)+\varphi(\nabla n_2 -n_2\nabla v_2),
\end{gather*}
we see that $\varphi$ satisfies
\begin{equation}\lb{ddewph}
\begin{aligned}
&\langle (n_2\varphi)', \phi_1 \rangle
 +\int_\Om \left\{n_1(\nabla \log(1+\varphi)-\nabla \eta)+\varphi(\nabla n_2 -n_2\nabla v_2)\right\} \cdot \nabla \phi_1 dx \\
&\quad +\int_\Om (R(n_1,p_1)-R(n_2,p_2)) \phi_1 dx=0.
\end{aligned}
\end{equation}
Similarly, $\psi$ solves
\begin{equation*}
\begin{aligned}
&\langle (p_2\psi)', \phi_2 \rangle
 +\int_\Om \left\{p_1(\nabla \log(1+\psi)+\nabla \eta) + \psi (\nabla p_2 +p_2\nabla v_2)\right\} \cdot \nabla \phi_2 dx \\
&\quad +\int_\Om (R(n_1,p_1)-R(n_2,p_2)) \phi_2 dx=0,
\end{aligned}
\end{equation*}
where $\phi_2 \in H^1_D(\Om)$.
We see from \eqref{ddewv} that $\eta$ satisfies
\begin{equation}\lb{ddewe}
\ve \int_\Om \nabla \eta \cdot \nabla \phi_3 dx+\int_{\Gamma_N} b\eta \phi_3 dS
 =-\int_\Om (n_2 \varphi -p_2 \psi) \phi_3 dx
\end{equation}
for all $\phi_3 \in H^1_D(\Om)$.

Now we derive an equality on the relative entropy of solutions.
\begin{lem}\lb{Diff1}
There holds that
\begin{align}
\frac{d}{dt}\left(\int_\Omega \mathcal{E}dx 
+\int_{\Ga_N} \tilde{\mathcal{E}} dS\right)
+\int_\Omega \mathcal{D} dx
=\int_\Omega (\mathcal{K}+\mathcal{L}+\mathcal{M}) dx,
\label{eform}
\end{align}
where
\begin{align*}
\mathcal{E}&:=n_2 \int_0^\varphi \log(1+y) dy
+p_2 \int_0^\psi \log(1+y) dy
+\frac{\ve}{2}|\nabla \eta|^2,
\notag \\ 
\tilde{\mathcal{E}}&:=\frac{b}{2}\eta^2,
\notag \\ 
\mathcal{D}&:=n_1 \left| \nabla \log (1+\varphi)-\nabla \eta \right|^2 
 +p_1 \left| \nabla \log (1+\psi)+\nabla \eta \right|^2,
\notag \\ 
\mathcal{K}&:=(R(n_2,p_2)-R(n_1,p_1))\left( \log (1+\varphi) +\log (1+\psi) \right),
\notag \\
\mathcal{L}&:=R(n_2,p_2) \left( \varphi-\log (1+\varphi) +\psi-\log (1+\psi) \right),
\notag \\
\mathcal{M}&:=\varphi (\nabla n_2-n_2 \nabla v_2) \cdot \nabla \eta
 -\psi (\nabla p_2+p_2 \nabla v_2) \cdot \nabla \eta.
\notag 
\end{align*}
\end{lem}
\begin{proof}
It suffices to show
\begin{align*}
\left.\left(\int_\Omega \mathcal{E}dx 
+\int_{\Ga_N} \tilde{\mathcal{E}} dS\right)\right|_{\tilde{t}}^{t}
+\int_{\tilde{t}}^{t} \int_\Omega \mathcal{D} dxdt
=\int_{\tilde{t}}^{t} \int_\Omega (\mathcal{K}+\mathcal{L}+\mathcal{M}) dxdt,
\end{align*}
since this gives the absolute continuity of 
$\int_\Omega \mathcal{E}dx +\int_{\Ga_N} \tilde{\mathcal{E}}dS$.

To make the following computation rigorous, 
we use a mollifier with respect to the time variable $t$ 
due to the insufficiency of the regularity of solutions.
We omit the argument since it is standard.

Choose $\phi_1=\log(1+\varphi)-\eta$ in \er{ddewph} to obtain
\begin{align}
&\langle n_1'-n_2',\log(1+\varphi) \rangle -\langle n_1'-n_2',\eta \rangle 
\notag\\
&\quad +\int_\Om \left\{n_1(\nabla \log(1+\varphi)-\nabla \eta)+\varphi(\nabla n_2 -n_2\nabla v_2)\right\} \cdot \nabla (\log(1+\varphi)-\eta)dx
\notag\\
&\qquad +\int_\Om (R(n_1,p_1)-R(n_2,p_2))(\log(1+\varphi)-\eta) dx=0. 
\lb{P1}
\end{align}
Let us rewrite the first and third terms of the left-hand side.
Noting $n_1-n_2=n_2\varphi$, we have
\begin{align*}
(n_1'-n_2')\log(1+\varphi) &= (n_2\varphi)' \log(1+\varphi)
\\
&=\left(\int_0^{n_2\varphi} \log\left(1+\frac{y}{n_2}\right)dy\right)' 
-\int_0^{n_2\varphi} \left(\log\left(1+\frac{y}{n_2}\right)\right)'dy 
\\
&=\left(n_2\int_0^{\varphi} \log\left(1+y\right)dy \right)'
+n_2'(\varphi-\log(1+\varphi)).
\end{align*}
Then, using \er{ddewn}, we arrive at
\begin{align*}
&\langle n_1'-n_2',\log(1+\varphi) \rangle \\
&=\left\{\int_\Om n_2\left(\int_0^{\varphi} \log\left(1+y\right)dy\right) dx\right\}'
-\int_\Om (\nabla n_2 -n_2\nabla v_2)\cdot\nabla(\varphi-\log(1+\varphi))dx
\\
&\quad -\int_\Om R(n_2,p_2) (\varphi-\log(1+\varphi))dx.
\end{align*}
One can rewrite the integrand of the third term on the left-hand side of \eqref{P1} as
\begin{align*}
&\left\{n_1(\nabla \log(1+\varphi)-\nabla \eta)+\varphi(\nabla n_2 -n_2\nabla v_2)\right\} \cdot \nabla (\log(1+\varphi)-\eta) 
\\
&=n_1|\nabla \log(1+\varphi)-\nabla \eta|^2
+(\nabla n_2 -n_2\nabla v_2)\cdot\nabla(\varphi-\log(1+\varphi))
-(\nabla n_2 -n_2\nabla v_2)\cdot\varphi\nabla\eta.
\end{align*}
From these, we obtain 
\begin{align}
&\left\{\int_\Om n_2\left(\int_0^{\varphi} \log\left(1+y\right)dy\right) dx\right\}'
-\langle n_1'-n_2',\eta \rangle
+\int_\Om n_1|\nabla \log(1+\varphi)-\nabla \eta|^2 dx
\notag \\
&=\int_\Om \left\{ R(n_2,p_2)\cdot\nabla(\varphi-\log(1+\varphi))
+(\nabla n_2 -n_2\nabla v_2)\cdot\varphi\nabla\eta \right.
\notag \\
&\qquad \qquad \left. -(R(n_1,p_1)-R(n_2,p_2))(\log(1+\varphi)-\eta) \right\}dx.
\lb{eform1}
\end{align}
Similarly, 
\begin{align}
&\left\{\int_\Om p_2\left(\int_0^{\psi} \log\left(1+y\right)dy\right) dx\right\}'
+\langle p_1'-p_2',\eta \rangle
+\int_\Om p_1|\nabla \log(1+\psi)+\nabla \eta|^2 dx
\notag \\
&=\int_\Om \left\{R(n_2,p_2)\cdot\nabla(\psi-\log(1+\psi))
-(\nabla p_2 +p_2\nabla v_2)\cdot\psi\nabla\eta \right.
\notag \\
&\qquad \qquad \left. -(R(n_1,p_1)-R(n_2,p_2))(\log(1+\psi)+\eta)\right\}dx.
\lb{eform2}
\end{align}
Note that \eqref{ddewe} yields
\begin{align}
\langle n_1'-n_2',\eta \rangle 
-\langle p_1'-p_2',\eta \rangle 
=\langle (n_2\varphi-p_2\psi)',\eta \rangle 
=- \left(\int_\Om \frac{\ve}{2}|\nabla\eta|^2dx
+\int_{\Ga_N} \frac{b}{2}\eta^2dS\right)'.
\lb{eform3}
\end{align}
Summing up \er{eform1}--\er{eform3} and integrating over $[\tilde{t},t]$
complete the proof.
\end{proof}


We remark that in the case that $(n_2,p_2,v_2)$ is a stationary solution $(N,P,V)$ satisfying \eqref{noflow}, 
the term $\cal K$ is nonpositive and the terms $\cal L$ and $\cal M$ are zero. 
Therefore it is easier to show its global stability. 
Among these terms, $\cal M$ is problematic to handle if $(n_2,p_2,v_2)$ does not satisfy \eqref{noflow}.
For this reason, we establish a new inequality in the following lemma.


\begin{lem}\lb{mpees}
There is a constant $C=C(\Omega,\Gamma_D,r,\ve)>0$ such that
\begin{equation}\lb{mpeineq}
|\varphi \nabla \eta |_2 +|\psi \nabla \eta |_2
 \le C(|n_2 \varphi |_2+|p_2 \psi |_2) (|\nabla \varphi |_2+|\nabla \psi |_2).
\end{equation}
\end{lem}
\begin{proof}
We note that by \eqref{ddewe},
\begin{equation}\label{elies}
|\eta|_\infty \le C|n_2 \varphi-p_2 \psi|_2.
\end{equation}
This follows from \eqref{ell-Linf}
with $h=(n_2 \varphi-p_2 \psi)/\ve$, $W_b=0$, $\tilde b=b/\ve$ and $\tilde g=0$.

Let us show \er{mpeineq}.
Taking $\phi_3=(\varphi^2+\psi^2) \eta$ in \eqref{ddewe} yields
\begin{align*}
&\ve \left( |\varphi \nabla \eta|_2^2 +|\psi \nabla \eta|_2^2 \right)
 +\int_{\Gamma_N} b(\varphi^2+\psi^2) \eta^2 dS \\
&=-2\ve \int_\Omega \eta (\varphi \nabla \varphi +\psi \nabla \psi )\cdot \nabla \eta dx
 -\int_\Om (n_2 \varphi -p_2 \psi) (\varphi^2+\psi^2) \eta dx.
\end{align*}
By the Schwarz inequality and \eqref{elies},
the first term of the right-hand side of this equality is estimated as
\begin{align*}
-2\ve \int_\Omega \eta (\varphi \nabla \varphi +\psi \nabla \psi )\cdot \nabla \eta
 &\le 2\ve \left( |\eta \nabla \varphi |_2^2 +|\eta \nabla \psi |_2^2 \right)
  +\frac{\ve}{2} \left( |\varphi \nabla \eta |_2^2 +|\psi \nabla \eta |_2^2 \right) \\
&\le 2\ve |\eta |_\infty^2 \left( |\nabla \varphi |_2^2 +|\nabla \psi |_2^2 \right)
 +\frac{\ve}{2} \left( |\varphi \nabla \eta |_2^2 +|\psi \nabla \eta |_2^2 \right) \\
&\le C\left( |n_2 \varphi |_2^2 +|p_2 \psi |_2^2 \right)
 \left( |\nabla \varphi |_2^2 +|\nabla \psi |_2^2 \right)
  +\frac{\ve}{2} \left( |\varphi \nabla \eta |_2^2 +|\psi \nabla \eta |_2^2 \right).
\end{align*}
The second term is handled as 
\begin{align*}
-\int_\Om (n_2 \varphi -p_2 \psi) (\varphi^2+\psi^2) \eta dx
 &\le |\eta|_\infty |n_2 \varphi-p_2 \psi|_2
  \left( |\varphi |_4^2 +|\psi |_4^2 \right) \\
&\le C\left( |n_2 \varphi |_2^2 +|p_2 \psi |_2^2 \right)
 \left( |\nabla \varphi |_2^2 +|\nabla \psi |_2^2 \right),
\end{align*}
where we have used 
\eqref{elies} and the Sobolev and Poincar\'{e} inequalities in deriving the last inequality.
Thus we obtain \eqref{mpeineq}.
\end{proof}

We are now in a position to prove Proposition~\ref{abDiff}.
\begin{proof}[Proof of Proposition~\ref{abDiff}]
In the proof, 
$c$ and $C$ stand for generic positive constants
depending only on $c_0$, $C_0$, $\Omega$, $\Gamma_D$, $\rho$, $r$, $\ve$ and $c_R$.
Define
\begin{equation*}
a=a(t):=|(n_2 p_2-1)(t)|_2^2 +|(\nabla n_2-n_2\nabla v_2)(t)|_2^2
 +|(\nabla p_2+p_2\nabla v_2)(t)|_2^2.
\end{equation*}
We claim that the desired inequalities are derived from the inequalities
\begin{gather} 
\int_\Om \mathcal{D}dx 
\geq c \left( |\nabla \varphi|_2^2+|\nabla \psi|_2^2\right),
\lb{Dineqfb}\\ 
\int_\Om (\mathcal{K}+\mathcal{L})dx \leq 
\mu \left( |\nabla \varphi|_2^2+|\nabla \psi|_2^2\right)
+\frac{Ca}{\mu^3} \left( |\varphi|_2^2+|\psi|_2^2\right),
\lb{KLineqfa}\\
\int_\Om \mathcal{M}dx \leq 
\mu \left( |\nabla \varphi|_2^2+|\nabla \psi|_2^2\right)
+\frac{Ca}{\mu} \left( |\varphi|_2^2+|\psi|_2^2\right),
\lb{Mineqfa}
\\
c\left( |\varphi|_2^2+|\psi|_2^2+\| \eta \|_1^2 \right) \leq
\int_\Omega \mathcal{E}dx +\int_{\Ga_N} \tilde{\mathcal{E}} dS \leq 
C \left( |\nabla \varphi|_2^2+|\nabla \psi|_2^2\right),
\lb{Eineqfa}
\end{gather}
where $\mu>0$ is an arbitrary number.
Let us verify this claim.
Substituting \er{Dineqfb}--\er{Mineqfa} into \er{eform}
and taking $\mu$ small enough,
we deduce that
\begin{equation*}
\frac{d}{dt}\left(\int_\Omega \mathcal{E}dx 
 +\int_{\Ga_N} \tilde{\mathcal{E}} dS\right)
  +c(|\nabla \varphi|_2^2+|\nabla \psi|_2^2 )
   \le C a (|\varphi|_2^2+|\psi|_2^2).
\end{equation*}
Applying \er{Eineqfa} to this inequality,
we have
\begin{align}\lb{esdd-1}
\frac{d}{dt}\left(\int_\Omega \mathcal{E}dx 
+\int_{\Ga_N} \tilde{\mathcal{E}} dS\right)
+ (c-Ca) \left( \int_\Omega\mathcal{E} dx 
+\int_{\Ga_N} \tilde{\mathcal{E}} dS \right)
+c \left( |\nabla \varphi|_2^2+|\nabla \psi|_2^2\right) 
\leq 0.
\end{align}
We now use \eqref{emineq} with $(n,p,v)=(n_2,p_2,v_2)$.
Integrating \eqref{emineq} and applying \eqref{asp1} give
\begin{equation}\label{esddta}
\int_s^t a(\tau)d\tau \le C+C\de (t-s)
\end{equation}
for all $t \ge s \ge \tilde t$.
Multiply \eqref{esdd-1} by $\exp(\int_s^t c-Ca(\tau) d\tau)$,
integrate the result and then use \eqref{esddta} to obtain
\begin{align}
&e^{c(t-\tilde t)} \left(\int_\Omega \mathcal{E} (t)dx+\int_{\Ga_N} \tilde{\mathcal{E}} (t)dS \right) 
 +\int_{\tilde t}^t e^{c(s-\tilde{t})} \left( |\nabla \varphi(s)|_2^2+|\nabla \psi(s)|_2^2 \right) ds 
\nonumber \\
&\le C\left( \int_\Omega \mathcal{E} (\tilde t)dx
 +\int_{\Ga_N} \tilde{\mathcal{E}} (\tilde t)dS \right) 
\label{esdd0}
\end{align}
provided that $\de$ is smaller than some number
$\de_0=\de_0(c_0,C_0,\Omega,\Gamma_D,\rho,r,\ve,c_R)>0$.
By \eqref{asp1},
we have
\begin{gather}
\int_\Omega \mathcal{E} (\tilde t)dx+\int_{\Ga_N} \tilde{\mathcal{E}} (\tilde t)dS \le C,
\label{esdd1} \\
|\nabla \varphi|_2^2+|\nabla \psi|_2^2
\ge c\left( |\nabla \log (1+\varphi)|_2^2+|\nabla \log (1+\psi)|_2^2\right).
\label{esdd2}
\end{gather}
Plugging \er{Eineqfa}, \eqref{esdd1} and \eqref{esdd2} into \eqref{esdd0} 
yields \eqref{dgsineq} and \eqref{ddgsineq} as claimed.

We complete the proof by showing \eqref{Dineqfb}--\eqref{Eineqfa}.
First let us show \er{Dineqfb}.
From the inequality $|a-b|^2 \ge |a|^2/2-|b|^2$ $(a,b \in \mathbb{R}^3)$ 
and \er{asp1}, we have
\begin{align}
\int_\Omega \mathcal{D} dx &\ge \frac{1}{2}\int_\Omega \left( \frac{n_2^2}{n_1}|\nabla \varphi |^2
 +\frac{p_2^2}{p_1}|\nabla \psi |^2 \right) dx
  -\int_\Omega (n_1+p_1)|\nabla \eta |^2 dx \nonumber \\
&\ge c\int_\Omega ( |\nabla \varphi |^2 +|\nabla \psi |^2 ) dx
 -C\int_\Omega |\nabla \eta |^2 dx. \lb{mdles}
\end{align}
To estimate $|\nabla \eta|_2$,
we take $\phi_3=\eta$ in \eqref{ddewe}.
Then
\begin{equation}\label{eliese}
\varepsilon \int_\Omega |\nabla \eta|^2dx +\int_{\Gamma_N} b\eta^2dS
 =\int_\Omega (-n_2\varphi +p_2\psi) \eta dx.
\end{equation}
By the fact that $a \log (1+a) \ge 0$ $(a>-1)$,
the Schwarz inequality and \er{asp1},
the integrand of the right-hand side of this equality is estimated as
\begin{align*}
(-n_2\varphi +p_2\psi) \eta &\le n_2 \varphi \left( \log (1+\varphi) -\eta \right)
 +p_2 \psi \left( \log (1+\psi) +\eta \right) \\
&\le \tilde \mu (\varphi^2+\psi^2)
 +\frac{C}{\tilde \mu} \left\{  \left( \log (1+\varphi) -\eta \right)^2 
  +\left( \log (1+\psi) +\eta \right)^2 \right\},
\end{align*}
where $\tilde \mu>0$ is an arbitrary constant.
Plugging this into \eqref{eliese} and then using the Poincar\'{e} inequality and \er{asp1},
we deduce that
\begin{align*}
&\int_\Omega |\nabla \eta|^2dx
 \\
 &\le C\tilde \mu \int_\Omega (|\nabla \varphi|^2 +|\nabla \psi|^2) dx 
  +\frac{C}{\tilde \mu} \int_\Om \left\{ \left| \nabla \log (1+\varphi) -\nabla \eta \right|^2 
   +\left| \nabla \log (1+\psi) +\nabla \eta \right|^2 \right\} dx \\
&\le C\tilde \mu \int_\Omega (|\nabla \varphi|^2 +|\nabla \psi|^2) dx 
 +\frac{C}{\tilde \mu} \int_\Omega \mathcal{D} dx.
\end{align*}
Substituting this into \eqref{mdles} and
choosing $\tilde \mu$ appropriately small give \eqref{Dineqfb}.

Next we derive \er{KLineqfa} and \er{Mineqfa}.
Note that
\begin{equation*}
\mathcal{K}=-c_R \frac{n_1p_1-n_2p_2}{n_1+p_1+2}\log \frac{n_1p_1}{n_2p_2}
 +c_R \frac{(n_2p_2-1)(n_2\varphi+p_2\psi)}{(n_1+p_1+2)(n_2+p_2+2)}
  (\log (1+\varphi)+ \log (1+\psi)).
\end{equation*}
Since the first term of the right-hand side of this equality is nonpositive,
we have
\begin{equation*}
\mathcal{K} \le c_R \frac{(n_2p_2-1)(n_2\varphi+p_2\psi)}{(n_1+p_1+2)(n_2+p_2+2)}
 (\log (1+\varphi)+ \log (1+\psi)) \le C|n_2p_2-1|(\varphi^2+\psi^2).
\end{equation*}
It is elementary to show that $\mathcal{L} \le C|n_2p_2-1| (\varphi^2 +\psi^2)$.
Hence, by the H\"older, Sobolev, Poincar\'e and Young inequalities, we have
\begin{align*}
\int_\Omega (\mathcal{K}+\mathcal{L}) dx 
&\le |n_2p_2-1|_2 \left( |\varphi|_2^{1/2} |\varphi|_6^{3/2}
+|\psi|_2^{1/2} |\psi|_6^{3/2}\right) 
\\
&\le C|n_2p_2-1|_2 \left( |\varphi|_2^{1/2} |\nabla \varphi|_2^{3/2}
+|\psi|_2^{1/2} |\nabla \psi|_2^{3/2} \right)
\\
&\le \mu \left( |\nabla \varphi|_2^{2}+|\nabla \psi|_2^{2} \right)
+\frac{C}{\mu^3} |n_2p_2-1|_2^3 \left( |\varphi|_2^2+|\psi|_2^2 \right).
\end{align*}
Here $\mu>0$ is an arbitrary number.
Owing to \er{asp1}, the last term can be estimated as
\begin{equation*}
|n_2p_2-1|_2^3 \left( |\varphi|_2^2+|\psi|_2^2 \right)
 \le C|n_2p_2-1|_2^2 \left( |n_2 \varphi|_2^2+|p_2 \psi|_2^2 \right).
\end{equation*}
Therefore, \er{KLineqfa} is proved.

The inequality \er{Mineqfa} is verified 
by applying the H\"{o}lder and Schwarz inequalities together with \eqref{mpeineq} as
\begin{align*}
\int_\Omega \mathcal{M} dx 
 &\le |\nabla n_2-n_2 \nabla v_2|_2 |\varphi \nabla \eta |_2
  +|\nabla p_2+p_2 \nabla v_2 |_2 |\psi \nabla \eta |_2 \\
&\le Ca^{1/2} \left( |n_2 \varphi |_2+|p_2 \psi |_2 \right)
 \left( |\nabla \varphi |_2+|\nabla \psi |_2 \right) \\
&\le \mu \left( |\nabla \varphi |_2^2+|\nabla \psi |_2^2 \right)
 +\frac{Ca}{\mu}\left( |n_2 \varphi |_2^2+|p_2 \psi |_2^2 \right).
\end{align*}

Finally we prove \er{Eineqfa}. 
It is easily seen from \er{asp1} and \eqref{eliese} that
\begin{gather*}
c (\varphi^2 +\psi^2 ) \le n_2 \int_0^\varphi \log(1+y) dy
 +p_2 \int_0^\psi \log(1+y) dy \le C (\varphi^2 +\psi^2), \\
\varepsilon \int_\Omega |\nabla \eta|^2dx +\int_{\Gamma_N} b\eta^2dS
 \le C\left( |\varphi|_2^2+|\psi|_2^2\right).
\end{gather*}
Hence we have
\begin{equation*}
c\left( |\varphi|_2^2+|\psi|_2^2+|\nabla \eta|_2^2 \right) 
\leq \int_\Omega \mathcal{E} dx +\int_{\Ga_N} \tilde{\mathcal{E}} dS 
\leq C\left( |\varphi|_2^2+|\psi|_2^2 \right).
\end{equation*}
We thus obtain \er{Eineqfa} by applying the Poincar\'e inequality 
to the right-hand side of this inequality.
The proof is complete.
\end{proof}

\section{Time-periodic solutions}\lb{TPS}
This section is devoted to the proof of Theorem~\ref{tpsthm}
stating the unique existence and global stability of time-periodic solutions.

\begin{proof}[Proof of Theorem~\ref{tpsthm}]
Throughout the proof, $c$ and $C$ denote generic positive constants
depending only on $c_0$, $C_0$, $\Omega$, $\Gamma_D$, $\rho$, $r$, $\ve$ and $c_R$.
Furthermore, we assume that $\theta<\theta_0$ and $\de<\min\{1,\de_0\}$, 
where $\theta_0$ and $\de_0$ are given 
in Theorem~\ref{uesthm} and Proposition~\ref{abDiff}, respectively.

First we show the uniqueness of time-periodic solutions.
Suppose that $(n_{*1},p_{*1},v_{*1})$ and $(n_{*2},p_{*2},v_{*2})$ 
are time-periodic solutions of \eqref{ddeq}.
Then Theorem~\ref{uesthm} ensures that
\begin{equation*}
\hat C^{-1} \le n_{*1},p_{*1},n_{*2},p_{*2} \le \hat C \quad \text{in } \bbr \times \Om,
\end{equation*} 
where $\hat C=\hat C(c_0,C_0,\Omega,\Gamma_D,\rho, r,\ve,c_R)$ 
is the constant being in \eqref{gsulb}.
Hence we see from Proposition~\ref{abDiff} that 
\begin{equation*}
|(n_{*1}-n_{*2})(t)|_2+|(p_{*1}-p_{*2})(t)|_2+\| (v_{*1}-v_{*2})(t) \|_1 \le C e^{-c (t-\tilde t)}
\end{equation*}
for all $t \ge \tilde t$, where $\tilde t \in \bbr$ can be chosen arbitrarily.
By letting $\tilde t \to -\infty$,
we obtain $(n_{*1},p_{*1},v_{*1})=(n_{*2},p_{*2},v_{*2})$, which establishes the uniqueness.

Next we investigate the existence of time-periodic solutions.
To this end, we fix a solution $(n,p,v)$ of \eqref{ddeq}.
From Theorem~\ref{uesthm}, 
we can choose $t_0$ such that
\begin{equation}\lb{npulb}
(2\hat C)^{-1} \le n,p \le 2\hat C \quad \text{in } (t_0,\infty) \times \Om.
\end{equation} 
Take an integer $k_0$ satisfying $k_0 T_*>t_0$
and define a sequence $\{ (n_k,p_k,v_k)\}_{k=0}^\infty$ by
\begin{equation*}
(n_k,p_k,v_k)(t,x):=(n,p,v)(t+(k_0+k)T_*,x),
 \quad (t,x) \in [-kT_*,\infty) \times \overline{\Om}.
\end{equation*}
Owing to \eqref{npulb} and the fact that $g$ and $V_{b}$ are periodic with period $T_*$, 
we see that $(n_k,p_k,v_k)$ solves \eqref{ddeq} for $I=(-kT_*,\infty)$ and satisfies
\begin{equation}\lb{npkulb}
(2\hat C)^{-1} \le n_k,p_k \le 2\hat C \quad \text{in } (-kT_*,\infty) \times \Om.
\end{equation} 
We can apply Proposition~\ref{abDiff}
with $(n_1,p_1,v_1)=(n_l,p_l,v_l)$, $(n_2,p_2,v_2)=(n_k,p_k,v_k)$ and $\tilde t=-kT_*$
to obtain
\begin{gather}
|(n_l-n_k)(t)|_2+|(p_l-p_k)(t)|_2+\| (v_l-v_k)(t) \|_1 \le C e^{-c (t+kT_*)},
\lb{dbnlk} \\
\int_{-kT_*}^t e^{c (s+kT_*)} \left( \left| \nabla \log \frac{n_l}{n_k} (s)\right|_2^2
+\left| \nabla \log \frac{p_l}{p_k}(s)\right|_2^2 \right) ds \le C,
\lb{drnpe}
\end{gather}
where $l>k$.
In particular, from \eqref{dbnlk}, there exists 
$(n_*,p_*,v_*) \in C(\bbr;L^2(\Om)) \times C(\bbr;L^2(\Om)) \times C(\bbr;H_D^1(\Om))$
such that
\begin{equation}\lb{converge1}
n_k \to n_*, p_k \to p_* \text{ in } C_{\text{loc}} (\bbr;L^2(\Om)), 
\quad
v_k \to v_* \text{ in } C_{\text{loc}} (\bbr;H^1(\Om))
\end{equation}
as $k \to \infty$.
Note that the limit $(n_*,p_*,v_*)$ is independent of the choice of $(n,p,v)$
used to define the sequence $(n_k,p_k,v_k)$,
since we have shown the uniqueness of time-periodic solutions.

Let us prove that $(n_*,p_*,v_*)$ is a time-periodic solution of \eqref{ddeq}
by checking the conditions (i)--(iv) in Definition~\ref{DOS}.
We see from \eqref{npkulb} and \eqref{converge1} that
\begin{equation*}
(2\hat C)^{-1} \le n_*,p_* \le 2\hat C \quad \text{in } \bbr \times \Om,
\end{equation*} 
which particularly gives the condition (ii).
By the definition of $(n_k,p_k,v_k)$, 
we have 
\begin{equation*}
(n_k,p_k,v_k)(t+T_*,x)=(n_{k+1},p_{k+1},v_{k+1})(t,x).
\end{equation*}
Hence letting $k \to \infty$ yields the condition (iv).
To check the conditions (i) and (iii),
we show that
\begin{equation}\lb{converge2}
n_k\nabla v_k \to n_*\nabla v_*, \
 p_k\nabla v_k \to p_*\nabla v_*, \
  R(n_k,p_k) \to R(n_*,p_*) \text{ in } L^2_{\text{loc}}(\bbr;L^2(\Om))
\end{equation}
as $k \to \infty$.
The convergence of $\{ n_k\nabla v_k\}$ follows from 
Lemma~\ref{A4} with $f_k=n_k$ and $g_k=\nabla v_k$.
In the same way,
we have $p_k\nabla v_k \to p_*\nabla v_*$.
By a simple calculation, one can check that 
$|\D R/\D n(n.p)|$, $|\D R/\D p(n.p)| \le c_R$,
and hence $|R(n_k,p_k)-R(n_*,p_*)| \le c_R (|n_k-n_*|+|p_k-p_*|)$.
This inequality and \eqref{converge1} imply that 
$R(n_k,p_k) \to R(n_*,p_*)$ in $L^2_{\text{loc}}(\bbr;L^2(\Om))$.
Thus $\eqref{converge2}$ is verified.
By \eqref{drnpe},
we deduce that $\{ \nabla \log n_k \}$ and $\{ \nabla \log p_k \}$ 
are Cauchy sequences in $L^2(J;L^2(\Om))$ for any bounded interval $J \subset \bbr$.
From this, \eqref{npkulb} and \eqref{converge1},
we can apply Lemma~\ref{A4} to conclude that
\begin{equation}\label{converge3}
\{ \nabla n_k \}=\{ n_k \nabla \log n_k \} \text{ and } \{ \nabla p_k \}=\{ p_k \nabla \log p_k \}
 \text{ are convergent in } L^2_{\text{loc}} (\bbr;L^2(\Om)).
\end{equation}
Note that $n_l-n_k$ satisfies
\begin{align*}
\langle n_l'-n_k', \phi_1 \rangle
 =-\int_\Om \left\{ (\nabla n_l -\nabla n_k) \cdot \nabla \phi_1
  -(n_l \nabla v_l -n_k \nabla v_k) \cdot \nabla \phi_1
   +(R(n_1,p_1)-R(n_2,p_2)) \phi_1 \right\} dx
\end{align*}
for all $\phi_1$.
This together with \eqref{converge2} and \eqref{converge3} gives
\begin{equation*}
\|n_l'-n_k'\|_{H^1_D(\Om)^*} \le |\nabla n_l-\nabla n_k|_2
 +|n_l \nabla v_l -n_k \nabla v_k|_2 +|R(n_l,p_l)-R(n_k,p_k)|_2
  \to 0 \quad (k,l \to \infty),
\end{equation*}
and therefore
\begin{equation}\label{converge4}
\{n_k'\} \text{ and } \{ p_k' \} \text{ are convergent in } L^2_{\text{loc}}(\bbr;H^1_D(\Om)^*).
\end{equation}
From \eqref{converge1}, \eqref{converge3} and \eqref{converge4}, 
we see that derivatives
$\nabla n_*,\nabla p_* \in L^2_{\text{loc}}(\bbr;L^2(\Om))$
and $n_*',p_*' \in L^2_{\text{loc}}(\bbr;$ $H^1_D(\Om)^*)$ exist, and
\begin{equation}\label{converge5}
\nabla n_k \to \nabla n_*, \nabla p_k \to \nabla p_*
 \text{ in } L^2_{\text{loc}} (\bbr;L^2(\Om)), \quad
  n_k' \to n_*', p_k' \to p_*' \text{ in } L^2_{\text{loc}} (\bbr;H^1_D(\Om)^*)
\end{equation}
as $k \to \infty$.
The condition (i) is therefore verified.
Furthermore, from \eqref{converge1}, \eqref{converge2}, \eqref{converge5} 
and the fact that $(n_k,p_k,v_k)$ satisfies the condition (iii),
we see that $(n_*,p_*,v_*)$ also satisfies the condition (iii).
Consequently, we have proved the existence of time-periodic solutions.

It remains to show \eqref{contps}.
By taking $k=0$ and letting $l \to \infty$ in \er{dbnlk},
we have
\begin{equation*}
|n(t+k_0 T_*)-n_*(t)|_2+|p(t+k_0 T_*)-p_*(t)|_2+\|v(t+k_0 T_*)-v_*(t)\|_1 \leq C e^{-c t}.
\end{equation*}
This together with the fact that $(n_*,p_*,v_*)$ is periodic with period $T_*$ gives
\begin{equation*}
|(n-n_*)(t)|_2+|(p-p_*)(t)|_2+\|(v-v_*)(t)\|_1 \leq C e^{-c t}.
\end{equation*}
We thus obtain \eqref{contps},
and the proof is complete. 
\end{proof}

\begin{appendix}
\section{Appendix}

This section provides the proofs of Lemmas \ref{eespro}--\ref{A4}.

\begin{proof}[Proof of Lemma \ref{eespro}]
The inequality \eqref{ell-Linf} can be verified by the same argument
as in \cite[Theorem~2, Remark~6]{GG86} or \cite[Theorem~2.5]{D},
and therefore we omit its proof.

It remains to show \eqref{ell-W11Lp}.
In what follows, $C$ denotes a generic positive constant
depending only on $\Omega,\Gamma_D,\al$ and $\be$.
We set ${\cal W}:=w-W_{b}$ and 
$k:=|h|_{1}+|\tilde g|_{1,\Gamma_{N}}+(1+|\tilde b|_{4/3,\Gamma_{N}})\|W_{b}\|_{1}$.
It suffices to prove that
\begin{equation}\label{W11Lp2}
|\nabla {\cal W}|_\al \le Ck
\quad \mbox{and} \quad
|{\cal W}|_{\be,\Gamma_N} \le Ck.
\end{equation}
By the condition $1 \le \al<3/2$,
we can take $\ga>0$ such that
\begin{equation*}
\kappa:=\frac{(1+\ga) \al}{2-\al}<3.
\end{equation*} 
Let $\kappa'>3/2$ and $\be'>2$ 
denote the H\"{o}lder conjugates of $\kappa$ and $\be$, respectively.
We first verify that for all $F \in L^{\kappa'}(\Om)$ and $G \in L^{\be'}(\Gamma_{N})$,
\begin{equation}
\left| \int_\Omega F {\cal W} dx+\int_{\Gamma_N} G {\cal W} dS \right|
 \leq Ck(|F|_{\kappa'}+|G|_{\be',\Gamma_{N}}).
\label{FGineq}
\end{equation}
To this end, we employ a solution  $\varphi \in H^{1}_{D}(\Omega)$ of the problem
\begin{gather}\label{varphi0}
\int_\Omega \nabla \varphi \cdot \nabla \tilde \phi dx
 +\int_{\Gamma_N} \tilde{b} \varphi \tilde \phi dS
  =\int_\Omega F \tilde \phi dx+\int_{\Gamma_N} G \tilde \phi dS
   \quad \text{for $\tilde \phi \in H^{1}_{D}(\Omega)$}.
\end{gather}
Owing to \eqref{ell-H1} and \eqref{ell-Linf}, 
we see that $\varphi$ satisfies
\begin{gather}\label{varphi1}
\|\varphi\|_{1}+|\varphi|_{\infty} \leq C(|F|_{\kappa'}+|G|_{\be',\Gamma_{N}}).
\end{gather}
Taking $\phi=\varphi$ in \eqref{weakeq} and $\tilde \phi={\cal W}$ in \eqref{varphi0},
and then combining the resulting equalities, 
we deduce that
\begin{align*}
\left| \int_\Omega F {\cal W} dx+\int_{\Gamma_N} G {\cal W} dS\right|
&=\left| -\int_\Om \nabla W_b \cdot \nabla \varphi dx -\int_\Om h\varphi dx 
 -\int_{\Gamma_N} ({\tilde b}W_b -{\tilde g}) \varphi dS \right|
\\
&\leq \| W_{b} \|_{1} \|\varphi\|_{1} +|h|_{1} |\varphi|_{\infty}
 +C(|\tilde b|_{4/3,\Gamma_{N}}\| W_{b} \|_{1} +|\tilde g|_{1,\Gamma_{N}})|\varphi|_{\infty}
\\
& \leq Ck(|F|_{\kappa'}+|G|_{\be',\Gamma_{N}}),
\end{align*}
where we have used \eqref{tsineq} and \eqref{varphi1}
in deriving the first and second inequalities, respectively.

From \eqref{FGineq} and the duality of $L^p$ spaces,
we see that
\begin{equation}
|{\cal W}|_\kappa \le Ck,
\qquad
|{\cal W}|_{\be,\Gamma_{N}} \le Ck.
\label{Wineq0}
\end{equation}
In particular, the latter inequality of \eqref{W11Lp2} holds.
What is left is to prove the former one.
The H\"older inequality gives
\begin{equation}
\int_\Omega |\nabla {\cal W}|^{\alpha} dx
 \leq \left( \int_\Omega |\nabla {\cal W}|^{2} \frac{dF}{ds}({\cal W}) dx \right)^{\!\!\! \alpha/2} 
  \left\{ \int_\Omega \left( \frac{dF}{ds}({\cal W})\right)^{-\al/(2-\al)} dx
   \right\}^{\!\!\! 1-\alpha/2},
\label{Walineq}
\end{equation}
where
\begin{equation*}
F(s):=\int_0^s \frac{k^{\ga+1}}{k^{\ga+1}+|\tilde s|^{\ga+1}} d\tilde s.
\end{equation*}
It is easily seen that
\begin{equation*}
|F(s)| \le \left( \int_0^\infty \frac{1}{1+|\tilde s|^{\ga+1}} d\tilde s \right) k,
\qquad
0 \le \frac{dF}{ds}(s) \le 1.
\end{equation*}
Taking $\phi=F({\cal W}) \in H^1_D(\Om)$ in \eqref{weakeq} and using the above inequalities,
we have
\begin{align*}
&\int_\Omega |\nabla {\cal W}|^{2} \frac{dF}{ds}({\cal W}) dx
 +\int_{\Gamma_N} \tilde{b} {\cal W} F({\cal W}) dS
\\
&=-\int_{\Omega} \left( \nabla W_{b} \cdot  \nabla  {\cal W} \right) \frac{dF}{ds}({\cal W}) dx
 -\int_\Omega h F({\cal W}) dx -\int_{\Gamma_N} (\tilde{b}W_{b}-\tilde{g}) F({\cal W}) dS  
\\
&\leq \frac{1}{2} \int_\Omega |\nabla {\cal W}|^{2} \frac{dF}{ds}({\cal W})^2 dx
 +\frac{1}{2} \int_\Omega |\nabla W_{b}|^{2} dx +Ck |h|_1
  +Ck(|\tilde b|_{4/3,\Gamma_{N}}\| W_{b} \|_{1} +|\tilde g|_{1,\Gamma_{N}})
\\
&\leq \frac{1}{2} \int_\Omega |\nabla {\cal W}|^{2} \frac{dF}{ds}({\cal W}) dx
+Ck^{2}.
\end{align*}
Since the second term of the left-hand side is nonnegative,
we arrive at
\begin{equation*}
\int_\Omega |\nabla {\cal W}|^{2} \frac{dF}{ds}({\cal W}) dx \le Ck^2.
\end{equation*}
It follows from \eqref{Wineq0} that
\begin{equation*}
\int_\Omega \left( \frac{dF}{ds}({\cal W})\right)^{-\al/(2-\al)} dx
 =\int_\Omega \left\{ 1+\left( \frac{|{\cal W}|}{k}\right)^{\ga+1}\right\}^{\al/(2-\al)} dx
  \le C+C\int_\Omega \left( \frac{|{\cal W}|}{k}\right)^{\kappa} dx \le C.
\end{equation*}
Substituting these inequalities into \eqref{Walineq}, we obtain \eqref{W11Lp2}. 
Thus  the lemma follows.
\end{proof}

\begin{proof}[Proof of Lemma \ref{Traces}]
The assertion (i) immediately follows from the boundedness of the trace operator 
from $W^{1,q}(\Omega)$ to $L^q(\partial \Omega)$ and the Poincar\'{e} inequality
$|f|_q \leq C|\nabla f|_q$.

One can show (ii) by combining the boundedness of the trace operator 
from $W^{1,1}(\Omega)$ to $L^1(\partial \Omega)$,
the H\"{o}lder inequality and the Sobolev embedding theorem 
$H^1(\Omega) \hookrightarrow L^6(\Omega)$.
Indeed, for $f \in H^1(\Om)$,
we have
\begin{align*}
|f|_{4,\Gamma_N} &=\left| |f|^4 \right|_{1,\Gamma_N}^{1/4}
  \le C\left( \left| |\nabla f| |f|^3 \right|_1^{1/4} +|f^4|_1^{1/4} \right)
\le C\left( |\nabla f|_2^{1/4} |f|_6^{3/4} +|f|_6 \right) \le C\| f\|_1.
\end{align*}

Let us show (iii).
We need only consider the case $q \ge 3/2$ owing to the fact that 
$L^{q_1}(\Gamma_N) \hookrightarrow L^{q_2}(\Gamma_N)$ for $q_1 \ge q_2$.
Using \eqref{tpineq} and the H\"{o}lder inequality, 
we have
\begin{equation}\lb{qbes}
|f|_{q,\Gamma_N} =\left| |f|^q \right|_{1,\Gamma_N}^{1/q}
  \le C\left| |\nabla f| |f|^{q-1} \right|_1^{1/q}
   \le C|\nabla f|_2^{\theta_1} |f|_6^{\theta_2} |f|_1^{\theta_3},
\end{equation}
where
\begin{equation*}
\theta_1=\frac{1}{q},
\quad
\theta_2=\frac{3}{5} \left( 2 -\frac{3}{q}\right),
\quad
\theta_3=\frac{1}{5} \left( \frac{4}{q} -1\right).
\end{equation*}
We note that $\theta_1+\theta_2+\theta_3=1$
and that the condition $3/2 \le q<4$ gives
$0 \le \theta_1+\theta_2<1$ and $0<\theta_3 \le 1$.
The Sobolev embedding theorem $H^1(\Omega) \hookrightarrow L^6(\Omega)$
and the Poincar\'{e} inequality yield $|f|_6 \le C|\nabla f|_2$,
and therefore we see from \eqref{qbes} that 
$|f|_{q,\Gamma_N} \le C|\nabla f|_2^{\theta_1+\theta_2} |f|_1^{\theta_3}$.
We thus obtain \eqref{trineq} by applying the Young inequality
to the right-hand side of this inequality. 
\end{proof}

\begin{proof}[Proof of Lemma~\ref{A5}]
It suffices to show that 
\begin{equation}\label{efeineq0}
a \log (da) +b \log (db)
 \le C\left\{ (a-b)^2+\frac{ab-1}{a+b+2}\log (ab) +1 \right\},
\quad a,b>0
\end{equation}
for some constant $C=C(d)>0$,
since
\begin{align*}
\int_A^a \log \frac{y}{A} dy +\int_B^b \log \frac{y}{B} dy
 &=a \log \frac{a}{A} +b \log \frac{b}{B} -(a+b) +A+B
\\
&\le a\log (da) +b\log (db) +2d.
\end{align*}
Suppose that \eqref{efeineq0} fails.
Then there exist sequences 
$\{ a_j\} \subset (0,\infty)$ and $\{ b_j\} \subset (0,\infty)$ such that
$F(a_j,b_j) \to \infty$ as $j \to \infty$,
where
\begin{equation*}
F(a,b):=\left( a \log (da) +b \log (db) \right)
 \left\{ (a-b)^2+\frac{ab-1}{a+b+2}\log (ab) +1 \right\}^{-1}.
\end{equation*}
It is easily seen that either $\{ a_j\}$ or $\{ b_j\}$ is unbounded.
Therefore, by taking a subsequence if necessary,
we may assume that $a_j+b_j \to \infty$ and $a_j/b_j \to l \in [0,\infty]$ as $j \to \infty$.
We first consider the case $l \in (1,\infty]$. 
Then, in particular, $b_j <a_j$ holds for large $j$.
Hence $a_j \to \infty$ as $j \to \infty$ and 
$b_j \log (db_j) \le a_j \log (da_j)$ for large $j$.
From these, we have
\begin{align*}
F(a_j,b_j) &\le 2a_j \log (da_j) \cdot (a_j-b_j)^{-2}
 =\frac{2\log (da_j)}{a_j} \cdot \left( 1-\frac{b_j}{a_j}\right)^{-2}
  \to 0 \quad (j \to \infty),
\end{align*}
which contradicts $\lim_{j\to \infty }F(a_j,b_j)=\infty$.
By a similar argument, 
we have a contradiction for $l \in [0,1)$. 
Next we assume that $l=1$. 
In this case, 
we have $a_j,b_j \to \infty$ as $j \to \infty$
and $\log (da_j),\log (db_j) $ $\le \log (a_j b_j)$ for large $j$.
It follows that
\begin{align*}
F(a_j,b_j) &\le (a_j+b_j) \log (a_j b_j) 
 \cdot \left( \frac{a_j b_j-1}{a_j+b_j+2}\log (a_j b_j)  \right)^{-1}
\\
&=\frac{(a_j/b_j+1)(a_j/b_j+1+2/b_j)}{a_j/b_j -1/b_j^2}
 \to 4 \quad (j \to \infty),
\end{align*}
a contradiction.
Thus we obtain \eqref{efeineq0}, 
and the proof is complete.
\end{proof}

\begin{proof}[Proof of Lemma~\ref{A3}]
We take $M_*>\max \{ 2\sigma, e\}$ 
such that the following hold for $M \ge M_*$:
\begin{equation}
\frac{\sqrt{\log M}}{(\sqrt{\log M}+1)^2}M \ge 1,
\quad
\frac{(\log M)^{3/2}}{(\sqrt{\log M}+1)^2} \ge 2h(M),
\quad
\log \frac{M}{2\sigma} \ge h(M),
\quad
\frac{\sqrt{\log M}-1}{\sqrt{\log M}+1} \ge \frac{1}{2}.
\label{lMineq}
\end{equation}
It is sufficient to show that
\begin{equation}
\frac{1}{a+b} \left\{ (a-b) H_M \left(\log \frac{a_{\sigma}}{b_{\sigma}}\right) 
 +\frac{ab-1}{a+b+2} \log (ab) \right\} \ge \frac{h(M)}{2}
\label{A3ineq0}
\end{equation}
for all $a,b>0$ and $M \ge M_*$ with $a+b \ge M$.
We divide the proof into the following three cases:
\begin{align*}
\text{(a)} \ \frac{1}{\sqrt{\log M}} \leq \frac{a}{b} \leq \sqrt{\log M}, 
\qquad
\text{(b)} \ \frac{a}{b}>\sqrt{\log M}, 
\qquad 
\text{(c)} \ \frac{a}{b}<\frac{1}{\sqrt{\log M}}.
\end{align*}

We consider the case (a).
It is easily seen that $z/(z+1)^2 \ge z_0/(z_0+1)^2$ 
if $z_0 \ge 1$ and $1/z_0 \le z \le z_0$.
Hence, by the first inequality of \eqref{lMineq}, 
we have
\begin{equation*}
ab=\frac{a/b}{(a/b+1)^2} (a+b)^2 \ge \frac{\sqrt{\log M}}{(\sqrt{\log M}+1)^2} M^2 \ge M.
\end{equation*}
In particular, we have $ab \ge 2$, and hence $ab-1 \ge ab/2$.
This together with the fact that $a+b \ge M \ge 2$ shows that
\begin{equation*}
\frac{ab-1}{(a+b)(a+b+2)} \ge \frac{ab}{4(a+b)^2} =\frac{a/b}{4(a/b+1)^2}
 \ge \frac{\sqrt{\log M}}{4(\sqrt{\log M}+1)^2}.
\end{equation*}
Thus
\begin{equation*}
\frac{ab-1}{(a+b)(a+b+2)} \log (ab)
 \ge \frac{\sqrt{\log M}}{4(\sqrt{\log M}+1)^2} \log M \ge \frac{h(M)}{2},
\end{equation*}
where we have used the second inequality of \eqref{lMineq}.
This shows that \eqref{A3ineq0} holds in this case.

It remains to examine the cases (b) and (c). 
We only consider the case (b),
since the case (c) can be dealt with in the same way.
It is seen that
\begin{equation}
a=\frac{1}{1+b/a} (a+b)>\frac{1}{1+1/\sqrt{\log M}} M \ge \frac{M}{2}.
\label{cbaineq}
\end{equation}
Note that this particularly gives $a \ge \sigma$.
Hence we have
\begin{align*}
\log \frac{a_\sigma}{b_\sigma} =\log \frac{a}{b_\sigma}
 \geq \min \left\{ \log \frac{a}{b}, \log \frac{a}{\sigma} \right\}
  \geq \min \left\{ h(M), \log \frac{M}{2\sigma} \right\} =h(M),
\end{align*}
where the second inequality follows from the conditions (b) and \eqref{cbaineq},
and the last equality follows from the third inequality of \eqref{lMineq}.
This together with the last inequality of \eqref{lMineq} gives
\begin{align*}
\frac{a-b}{a+b}H_M \left(\log \frac{a_{\sigma}}{b_{\sigma}}\right) =\frac{a/b-1}{a/b+1} h(M)
 \ge \frac{\sqrt{\log M}-1}{\sqrt{\log M}+1}h(M) \ge \frac{h(M)}{2}.
\end{align*}
Thus \eqref{A3ineq0} is verified, and the proof is complete. 
\end{proof}

\begin{proof}[Proof of Lemma~\ref{A4}]
Let $f \in L^2(E)$ and $g \in L^2(E)$ 
be the limits of $\{f_k\}$ and $\{g_k\}$, respectively. 
Then the assumption $|f_k| \le C$ implies that $|f| \le C$.
For $M>0$,
we have
\begin{align*}
\int_E (f_kg_k-fg)^2 dx
 &\le 2\int_E f_k^2 (g_k-g)^2 dx +2\int_E (f_k-f)^2 g^2 dx \\
&=2\int_E f_k^2 (g_k-g)^2 dx
 +2\int_{|g|<M} (f_k-f)^2 g^2 dx +2\int_{|g| \ge M} (f_k-f)^2 g^2 dx \\
&\le 2C^2 \int_E (g_k-g)^2 dx 
 +2M^2 \int_E (f_k-f)^2 dx +8C^2 \int_{|g| \ge M} g^2 dx,
\end{align*}
and hence
\begin{equation*}
\limsup_{k \to \infty} \int_E (f_kg_k-fg)^2 dx \le 8C^2 \int_{|g| \ge M} g^2 dx.
\end{equation*}
Thus the lemma follows by letting $M \to \infty$.
\end{proof}
\end{appendix}

\noindent
{\bf Acknowledge.} \
T. Kan was supported by JSPS KAKENHI Numbers 19K14574.
M. Suzuki was supported by JSPS KAKENHI Numbers 18K03364.

\end{document}